%% file: Titre_Article.tex
\documentclass[reqno]{amsart}
\usepackage{amsmath,amscd,mathrsfs}
\usepackage[hyphens]{xurl}
\usepackage[breaklinks,colorlinks=true,citecolor=blue,linkcolor=blue,urlcolor=blue]{hyperref}
\usepackage{amsthm,theoremref}
\usepackage{amssymb}
\usepackage[new]{old-arrows}
\usepackage{graphics}
\usepackage[all]{xy}

\usepackage{multicol}
\usepackage[refpage]{nomencl}

\swapnumbers

\newtheorem{thm}{Theorem}[section]

\newtheorem{prop}[thm]{Proposition}
\newtheorem{convention}[thm]{Convention}

\theoremstyle{definition}

\newtheorem{fact}[thm]{Fact}

\newtheorem{cor}[thm]{Corollary}

\newtheorem{Regles}[thm]{Rules}
\newtheorem{lem}[thm]{Lemma}

\newtheorem{defn}[thm]{Definition}

\newtheorem{factdef}[thm]{Fact/Definition}

\newtheorem{revision}[thm]{Revision}

\newtheorem{examples}[thm]{Examples}

\newtheorem{setup}[thm]{Set Up}

\newtheorem{claim}[thm]{Claim}
\newtheorem{subclaim}[thm]{Sub-Claim}

\newtheorem{not/clar}[thm]{Notation/Clarification}

\newtheorem{simplification}[thm]{Simplifications}

\newtheorem{defrev}[thm]{Definition/Revision}
\newtheorem{warning}[thm]{Warning}
\newtheorem{bonus}[thm]{Bonus}
\newtheorem{babyCrit}[thm]{Baby Criterion}
\newtheorem{PractCrit}[thm]{Practical Criteria}
\newtheorem{CheckList}[thm]{Check List}

\theoremstyle{remark}

\newtheorem{triviality}[thm]{Triviality}

\newtheorem{rmk}[thm]{Remark}

\numberwithin{equation}{section}

\def\dar[#1]{\ar@<2pt>[#1]\ar@<-2pt>[#1]}
  \entrymodifiers={!!<0pt,0.7ex>+}

\newcommand\Atop[2]{\genfrac{}{}{0pt}{}{#1}{#2}}

\begin{document}

\title{Flattening and algebraisation}
\author{Michael McQuillan}
\address{Universit\'a degli studi di Roma 'Tor Vergata' \& HSE University, Moscow}
\date{\today. Support provided within the framework of the HSE University basic research programme.}

\begin{abstract}
To, say, a proper algebraic or holomorphic space $X/S$, and a coherent sheaf ${\mathscr F}$ on $X$ we identify a functorial ideal, the fitted flatifier, blowing up sequentially in which leads to a flattening of the proper transform of ${\mathscr F}$. As such,  this is a variant on theorems of Raynaud \& Hironaka, but it's functorial nature allows its application to a flattening theorem for formal algebraic spaces or Artin champs. An immediate application is the (practical \& drop in) replacement anticipated in \cite{artin} for the definition of surjectivity therein. A more involved application is close to optimal algebraisation theorems for formal deformations. En passant, contrary to what is asserted in \cite{ega3}[Remarque 5.4.6], we give an example of an adic Noetherian formal scheme whose nil radical is not coherent and establish the equivalence conjectured therein between arbitrary algebraisability and that of the reduction.
\end{abstract}


\maketitle

\input{0Intro}

\vglue 1cm

\input{1Flatifier_In_General.tex}

\vglue 1cm

\input{2Pointed_Flatifier.tex}

\vglue 1cm

\input{3Compactification_General.tex}

\vglue 1cm

\input{4Compactifying_Flatificator.tex}

\vglue 1cm

\input{5Algebraisation_Formal_Deformations.tex}

\vglue 1cm

\input{6Appendix.tex}

\vglue 1cm

\bibliography{elvis}{}
\bibliographystyle{Gamsalpha}

\end{document}

%% file: 0Intro.tex
\section{Introduction}\label{Intro}

A priori Raynaud's proof of flattening by way of a sequence of blow ups, \cite{ray} [Th\'eor\`eme 5.2.2], makes such an essential use of algebraicity to extend from a local to a global centre, op. cit. 5.3.1, that Hironaka, \cite{hironaka}[Introduction], reckons his proof in the analytic case substantially different. Locally, however, both employ a centre defined by a certain coefficient ideal, which is often the same, but Hironaka's analysis, \cite{hironaka}[Theorem 1.14] is more precise and allows him to identify the said centre with the functorial notion of a pointed flatifier, \thref{flatP:defn1}. Nevertheless this doesn't immediately solve the problem since the global notion of a flatifier, \thref{flatG:rev1}, is only net, so, cf. \cite{hironaka}[Example 2], there may be no global closed sub-space restricting to the desired local centre. As such Hironaka's proof of the analytic flattening theorem is by way of a non-trivial double induction, op. cit. Theorems I$_n$ \& II$_n$, in order to exploit the local calculation. A more systematic use, however, of the flatifier as developed by Grothendieck in \cite{mur} reduces this to an easy induction, \thref{end:fact2}, of a couple of lines. Specifically, in the analytic category, the flatifier is not a priori compactifiable, and one uses, inductively, the flattening theorem together with \cite{ega4}[Th\'eor\`eme 11.4.1] (which could have been written for exactly this purpose) to conclude that a compactification exists, from which one finds a global centre, the fitted flatifier, \thref{end:defn1}, whose ideal is locally a product of the ideal of the pointed flatifier (of the \'etale local components where it isn't schematically dense) with some a priori unknown stuff, but, from the point of view of flatness, this is more than adequate to improve the situation on blowing up in it.

\smallskip

At which point the principle obstruction to proving a flattening theorem for formal schemes or algebraic spaces is that it should have sense. Specifically the correct statement is, already for schemes, not that one can achieve a flattening after a sequence, say, $\rho : \widetilde S \to S$, of blow ups, but that the flatifier over $\widetilde S$ is an everywhere dense closed embedding. As such if $S$ is reduced then the flatifier over $\widetilde S$ is $\widetilde S$, but if, say, $S$ were Artinian, then $\widetilde S = S$ and there is no possible improvement. Even, however, for affine formal, and indeed adic Noetherian, schemes the nil radical needn't be coherent, \thref{cor:nil2}, and one needs to impose a condition, \thref{con:defn1}, the contrary assertion in \cite{ega3}[Remarque 5.4.6] notwithstanding, which we have called {\it consistency}, and whence it's pertinent to make,

\begin{rmk}\thlabel{rmk:intro}
All algebraic spaces, and even Deligne Mumford champs, are consistent. In the formal category however it is a non-empty condition implied by, for example, all local rings excellent, \thref{con:ex1}, which in any geometrically meaningful situation, i.e. the trace is quasi-excellent, is always true by a theorem, \cite{gabber}, of Gabber. Better still for the purpose of flattening a still weaker condition {\it consistent in a Japanese way}, \thref{con:defn1}, is all that's required.
\end{rmk}

This said we arrive to the main,

\begin{thm}\thlabel{main}
Let $S = {\rm Spf} (A)$ be an adic Noetherian formal scheme which is universally consistent in a Japanese way, and $f : {\mathcal X} \to S$ a proper map from an adic formal champ in the smooth topology, then any coherent sheaf, ${\mathscr F}$, on ${\mathcal X}$ determines an ideal, ${\mathcal I}_S ({\mathscr F})$, the fitted flatifier, \thref{end:defn1}, no associated prime of which is minimal, respecting smooth base change, i.e. if $T \to S$ is formally, cf. \thref{hensel:factdef1}, smooth,
\begin{equation}
\label{intro1}
{\mathcal I}_S ({\mathscr F}) \vert_T = {\mathcal I}_T ({\mathscr F} \vert_T) \, .
\end{equation}
As such if $f : {\mathcal X} \to {\mathcal S}$ were a proper map of adic formal champ in the smooth topology and ${\mathscr F}$ is a coherent sheaf on ${\mathcal X}$, these ideals glue to an ideal ${\mathcal I}_{\mathcal S} ({\mathscr F})$ and a blow up, in the fitted flatifier,
\begin{equation}
\label{intro2}
 \rho_{\mathcal S} : {\mathcal S}^{\#} \to {\mathcal S}
\end{equation}
with exceptional divisor $D$, and, whence, 
for ${\mathscr H}_E^0$ torsion along the same
(rather than the formal support in $D$ functor) 
a sequence of blow ups and coherent sheaves,
\begin{equation}
\label{intro3}
{\mathcal S}_0 = {\mathcal S} \, , \ {\mathscr F}_0 = {\mathscr F} \, , \ {\mathcal S}_{i+1} = {\mathcal S}_i^{\#} \, , \ {\mathscr F}_{i+1} = \rho_S^* {\mathscr F} / {\mathscr H}_D^0 (\rho_S^* {\mathscr F})
\end{equation}
such that for some $n \gg 0$, the flatifier of ${\mathscr F}_n$ over ${\mathcal S}_n$ is an everywhere closed subspace with the same reduction, or, equivalently ${\mathscr F}_n$ is flat over $({\mathcal S}_n)_{\rm red}$, \thref{end:cor1}.
\end{thm}

Here for ease of of exposition we've supposed proper, but the necessary and sufficient condition for the existence of the flatifier, \thref{flatG:fact2}, is that the {\it non-flatness of ${\mathscr F}$ is universally pure}, \thref{flatG:def2}, and when, in the discrete case, this holds at every one of the steps in \eqref{intro3}, \thref{main} holds too,  although
such inductive purity isn't easily verfied without some properness, \cite{dave}[Example 1.21]. Irrespectively, neither purity nor properness are
necessary since \thref{main} implies any sub-case where a compactification is known to exist, e.g. $f$ a quasi-compact map of finite type between algebraic spaces, \cite{conrad}. In
the formal case, however, the matter is more subtle and it's appropriate to spell out the
issue:

\begin{rmk}\thlabel{rmk:intro2}
Just as in \cite{hironaka}, the openness of flatness is key to building the centres in \thref{main}, which is awkward since this is more naturally a statement about schemes rather than formal schemes, cf. \thref{flatG:lem1}--\thref{flatG:cor1}. Locally, on $\mathcal{X}$, this is a trivial variant of an argument of Reinhardt Kiehl, \cite{kieh}, in the holomorphic setting, which in turn is the only known proof for analytic sheaves -- critically Lemma IV.10 of the usual citation \cite{rubbish} asserts what has to be proved, i.e. middle of the proof, notation therein, $B(K,{\mathscr C})$ is acyclic. Similarly, the literature on purity, although not as bad as that on analytic flatness, could be better since if ${\mathscr F}$ were a pure module then its non-flatness would be pure too, \thref{A:formalpureT}, but there are plenty, \thref{flatG:rmk1}, of flat modules which aren't pure. Irrespectively, in the discrete case, purity translates the closed non-flat locus on $\mathcal{X}$ into a closed subspace of $S$, but exactly how to do this in the formal case without supposing the said locus proper is far from clear. 
\end{rmk}

All of this said, putting some functorial order to the literature on flattening was a by product of our motivation, namely, algebraisation of formal deformations of foliated surfaces of general type within a larger project of a Deligne-Mumford theorem for their moduli. Normally, for varieties of general type, this is an easy consequence, \cite{ega3}[Th\'eor\`eme 5.4.5] of the existence of canonical models, since they naturally have an ample line bundle. However, this is by no means the case for foliations, \cite{canmod}[Corollary IV.2.3], where one needs something like,

\begin{thm}\thlabel{two}(\thref{alg:real1})
Let ${\mathfrak X} \to S$ be a proper adic formal algebraic space over an adic Noetherian formal scheme which is universally consistent in a Japanese way, \thref{con:defn1}, and  for which 
every irrdeucible component ${\mathfrak X}_i$ admits
a non-empty $S$-flat  open subspace $U_i \hookrightarrow {\mathfrak X}$ 
together with a line sheaf ${\mathscr L}_i$
such that,

(a) On some, and whence any, modification, \thref{A:mod}, $\varphi_i : \widetilde{\mathfrak X}_i \to {\mathfrak X}_i$ with exceptional 
locus $E$, $\varphi_i^* \, {\mathscr L}_i$ modulo $E$-torsion is a nef line bundle $L_i$;

(b) For $X_i \hookrightarrow {\mathfrak X}_i$ the trace, the map afforded by global sections,
\begin{equation}
\label{aaalg3}
\mathrm{H}^0 (X_i , {\mathscr L}_i \mid_{X_i}) \otimes_{{\mathscr O}_{S_0}} {\mathscr O}_{X_i} \longrightarrow {\mathscr L}_i
\end{equation}
embeds a generic point of $X_i\cap U_i$;

\noindent then ${\mathfrak X}$ is algebraisable, i.e. it is the completion of a proper algebraic space $X/S$.
\end{thm}

The technique isn't to count sections, but their growth, \thref{setup:count1}, modulo powers of the maximal ideal, which, arguably gives, en passant, alternative and less ``schemy'' proofs of the invariance of the Euler characteristic, and upper semi-continuity theorem, \thref{count:fact2}, for the cohomology of coherent sheaves. It is, therefore, wholly necessary to first establish \thref{alg:lem0}, that a space is algebraisable iff its reduction is, as conjectured in \cite{ega3}[Remarque 5.4.6].

In the same vein, and indeed as a much more immediate corollary of \thref{main},
the Stein fatorisation theorems of \cite{ega3}[\S 4.3] can now be made to work
for formal schemes, so, as wholly anticipated in \cite{artin}[Remarks, post 1.7],
op. cit.'s (often difficult to verify) definition of surjectivity reduces to what
it should be,

\begin{thm}\label{three}
Let $f:{\mathfrak X} \to S$ be a proper adic formal algebraic space over an adic Noetherian formal scheme which is universally consistent in a Japanese way, \thref{con:defn1}, then $f$ is surjective in
the sense of Artin, \thref{A:mod}, iff $\mathscr{O}_{\mathcal{S}_{\mathrm{red}}}\to f_*\mathscr{O}_{\mathfrak{X}_{\mathrm{red}}}$ is injective.
\end{thm}
\smallskip

The manuscript was written while visiting Bures in February 2024, but its typesetting is the careful and precise work of C\'ecile, while Callum Spicer and David Ryhd made invaluable contributions to correcting some appalling errors.




%% file: 1Flatifier_In_General.tex
\section{The flatifier in general}\label{Flatifier}

We'll require to generalise from finite type in the discrete
sense to finite type in the formal sense, \thref{formal:rev1}, 
various well-known results beginning with,

\begin{lem}\thlabel{flatG:lem1}
Let $S = {\rm Spec} \, (A)$ be the spectrum of an $I$-adically complete Noetherian ring and $f : X = X^n_S \to S$ the spectrum of the ring of restricted $n$-dimensional formal power series, \thref{formal:rev1}, then for any coherent sheaf, $\mathscr{F}$, of $\mathscr{O}_X$-modules the set,
\begin{equation}
\label{flatG.1}
\{x \in X \mid \mathscr{F}_x \ \text{is} \ \mathscr{O}_{S,f(x)} \ \text{flat}\}
\end{equation}
is open in $X$.
\end{lem}

\begin{proof}
Following \cite{kieh}, immediately after his lemma 3, consider the diagram,
\begin{equation}
\label{flatG.2}
\xymatrix{
&X_S^{2n} \ar[ld]^-{p_2} \ar[dd]^{p_1} &&X \ar@{_{(}->}[ll]^-{\Delta_{X/S}} \\
X = X_S^n \ar[dd]_f \\
&X = X_S^n \ar[dl]^f \\
S
}
\end{equation}
where all of $f , p_1 , p_2$ are flat by \cite{matsumura}[Theorem 22.3], 
and \eqref{flatG.1} is topologically fibred, so, by \thref{flatG:rule3},
$$
\mathscr{F}_x \ \text{is} \ \mathscr{O}_{S,f(x)} \ \text{flat} \Longleftrightarrow (p_2^* \mathscr{F})_{\Delta (x)} \ \text{is} \ \mathscr{O}_{X,x} \ \text{flat}
$$
wherein the respective $\mathscr{O}_{X,x}$ module structures are by way of $p_1^*$. As such we may immediately conclude by \cite{kieh}[Satz 1].
\end{proof}

Which offers the immediate,

\begin{cor}\thlabel{flatG:cor1}
Let everything be as in \thref{flatG:lem1} but suppose more generally that $f : X = {\rm Spec} \, B \to S$ is only embedded in some $X_S^n$, i.e. $B/A$ is (topologically) of finite type, \thref{formal:rev1}, then the same conclusion holds.
\end{cor}

This suggests that we introduce the following,

\begin{defn}\thlabel{flatG:def1}
Let $B$ be of (topologically) finite type over an $I$-adically complete Noetherian ring $A$, and $M$ a finitely generated $B$ module, then $\Theta_{B/A} (M)$ is the ideal of the closed (reduced) subset complementary to the open set \eqref{flatG.1} for the $\mathscr{O}_S$-module $\mathscr{F}$ associated to $M$. Alternatively, therefore, if we employ the geometric notation of \thref{flatG:lem1} then we will write $\Theta_f (\mathscr{F})$ for the resulting sheaf of ideals.
\end{defn}

Observe that the ideals $\Theta$ satisfy,

\bigskip

\begin{Regles}\thlabel{flatG:rule1}
If,
\begin{equation}\label{flatG.3}
\xymatrix{
X \ar[rd]_h \ar[rr]_g &&Y \ar[ld]^f \\
&S
}
\end{equation}
is a commutative diagram of Noetherian affine schemes, wherein $\Gamma(X)$, resp. $\Gamma(Y)$, are (topologicaly) of finite type, \thref{formal:rev1}, over an $I$-adically complete Noetherian ring, $\Gamma(S)$, and $\mathscr{F}$ is a coherent sheaf of $\mathscr{O}_Y$-modules, then if $g$ is flat,
\begin{equation}
\label{flatG.4}
\sqrt{g^* \Theta_f (\mathscr{F})} = \Theta_h (g^* \mathscr{F}) \, .
\end{equation}
\end{Regles}

\begin{Regles}\thlabel{flatG:rule3}
If
\begin{equation}
\begin{CD}\label{flatG.444}
Y@<<{g}< X \\
@V{f}VV @VV{h}V \\
S@<{g}<< T
\end{CD}
\end{equation}
is a commutative diagram of Noetherian affine schemes, wherein $\Gamma(X)$, resp. $\Gamma(Y)$, are (topologically) of finite type, \thref{formal:rev1}, over an $I$-adically complete Noetherian ring, $\Gamma(S)$, and $\mathscr{F}$ is a coherent sheaf of $\mathscr{O}_Y$-modules, then if $g$ is flat, and \eqref{flatG:rule3} is topologically
fibred,
\begin{equation}
\label{flatG.44}
\sqrt{g^* \Theta_f (\mathscr{F})} = \Theta_h (g^* \mathscr{F}) \, .
\end{equation}
\end{Regles}

\begin{proof}
To begin with in \thref{flatG:rule1}, $\mathscr{F}$ is coherent, so,
\begin{equation}\label{base1}
g^* \mathscr{F}:= \mathscr{O}_X{\otimes}_{\mathscr{O}_Y}\mathscr{F}
= \mathscr{O}_X\widehat{\otimes}_{\mathscr{O}_Y}\mathscr{F}
\end{equation}
while, \thref{fact:flatR1}, flatness and topological flatness coincide for coherent modules,
 so $\Theta_h (g^* \mathscr{F})$ $\hookleftarrow g^* \Theta_f (\mathscr{F})$ in \eqref{flatG.3}. Conversely if for some points $y = g(x)$, lying over $s$, $\mathscr{F}_y$ is not a flat $\mathscr{O}_{S,s}$ module, then there is an ideal $J$ of the latter and a non-zero module $K$ such that,
\begin{equation}
\label{flatG.7}
0 \to K \to J \otimes_{\mathscr{O}_{S,s}} \mathscr{F}_y \twoheadrightarrow J \mathscr{F}_y \hookrightarrow \mathscr{F}_y
\end{equation}
is exact. We can however pull this back along the faithfully flat map $\mathscr{O}_{Y,y} \to \mathscr{O}_{X,x}$ to obtain,
\begin{equation}
\label{flatG.8}
0 \to K \otimes_{\mathscr{O}_{Y,y}} \mathscr{O}_{X,x} \ne 0 \to J \otimes_{\mathscr{O}_{S,s}} 
g^*\mathscr{F}_x  \hookrightarrow g^*\mathscr{F}_x
\end{equation}
so $\mathscr{F}_x$ isn't flat either. 

As to \thref{flatG:rule3}, since \eqref{flatG.444} is topologically fibred, the
upper horizontal is topologically flat, so it's flat tout court. Consequently, the
left hand side of \eqref{flatG.44} is equally,
\begin{equation}\label{onus2}
\Theta_{fg=gh} (g^* \mathscr{F}), 
\end{equation}
and the right hand side of \eqref{flatG.44} is included in the left. The
other direction is stability of (discrete) flatness under topological base change, 
and needs a little work. Specicially, 
for $\mathscr{F}$, resp. $\mathscr{H}$,
coherent on $Y$, resp. $T$,
\begin{equation}\label{base2}
g^* \mathscr{F}{\otimes}_{\mathscr{O}_T}\mathscr{H}
=g^* \mathscr{F}{\otimes}_{\mathscr{O}_X}h^*\mathscr{H} 
=g^* \mathscr{F}\widehat{\otimes}_{\mathscr{O}_X}h^*\mathscr{H}
=f_*\mathscr{F}\widehat{\otimes}_{\mathscr{O}_S}g_*\mathscr{H}
\end{equation}
which, more correctly, should be understood at the level of global sections
rather than sheaves.
In particular, continuing with the same abuse of notation, 
for $\mathscr{H}'\hookrightarrow \mathscr{H}$, any inclusion of
coherents on $T$, we have exact sequences,
\begin{equation}\label{base3}
\begin{split}
0&\rightarrow\widehat{\mathscr{N}}
\rightarrow g^*\mathscr{F}\widehat{\otimes}_{\mathscr{O}_T}\mathscr{H}'
\rightarrow
g^* \mathscr{F}\widehat{\otimes}_{\mathscr{O}_T}h^*\mathscr{H}\\
0&\rightarrow{\mathscr{N}}
\rightarrow g^*\mathscr{F}{\otimes}_{\mathscr{O}_T}\mathscr{H}'
\rightarrow
g^* \mathscr{F}{\otimes}_{\mathscr{O}_T}h^*\mathscr{H}
=f_*\mathscr{F}{\otimes}_{\mathscr{O}_S}g_*\mathscr{H}
\end{split}
\end{equation}
in which the topologically coherent sequence of $\mathscr{O}_X$ modules 
on the top line is the completion
of, the very far from coherent, sequence of $\mathscr{O}_T$ modules on the bottom, so
we need \cite{formal}[Claim 2.6] to make sense of this. In any case, 
for 
$Q$ a prime ideal of  $\Gamma(\mathscr{O}_Y)$ lying over
a prime ideal $P$ of  $\Gamma(\mathscr{O}_S)$, flatness
of $\mathscr{F}_Q$ over $\mathscr{O}_{S,P}$ is equivalent
to anything in $\mathscr{N}$ being annihilated by something
in the complement of $Q$. Equally for any finite 
$\Gamma(\mathscr{O}_X)$ module the sub-module of
elements annihilated by an element 
of  $\Gamma(\mathscr{O}_Y)\backslash Q$
is closed, so once this is true of $\mathscr{N}$,
it's true of $\widehat{\mathscr{N}}$ too.
As such, to deduce that
$g^*\mathscr{F}$ is a flat map of local rings at any prime $Q'$
lying over $P'$ in   $\Gamma(\mathscr{O}_T)$, we
just take $\mathscr{H}'\hookrightarrow\mathscr{H}$ to be the
inclusion of the Zariski closure of an ideal in $\Gamma(\mathscr{O}_T)_{P'}$ 
in the structure sheaf.
\end{proof}

En passant we have, usefully, picked up

\begin{bonus}\thlabel{flatG:bonus1}
The various hypothesis in \thref{flatG:rule1} and \thref{flatG:rule3} that $\Gamma(X)$, resp. $\Gamma (Y)$, was of finite type, \thref{formal:rev1}, over the $I$-adically complete Noetherian ring $\Gamma (S)$ or $\Gamma (T)$ were only used to guarantee that $\Theta_f (\mathscr{F})$ etc. existed. Alternatively, however, if for example, in \thref{flatG:lem1} we simply denote by $Z_f (\mathscr{F})$ the complement of \eqref{flatG.1} then we've actually established, independently of such hypothesis,
\begin{enumerate}
\item[(1)] In the situation of \eqref{flatG.3}: if $Z_f (\mathscr{F})$ is closed, and $g$ is flat, then $Z_h (g^* \mathscr{F}) = g^{-1} Z_f (\mathscr{F})$.
\item[(2)] Similarly, in the situation of \eqref{flatG.444}:if $Z_f (\mathscr{F})$ is closed, and $T\to S$ is flat, then $Z_h (g^* \mathscr{F}) = g^{-1} Z_f (\mathscr{F})$.
\end{enumerate}
\end{bonus}

Plainly if such considerations allow us to conclude that $Z_f (\mathscr{F})$ is closed then we continue to denote the resulting ideal by $\Theta_f (\mathscr{F})$. As such,

\begin{fact}\thlabel{flatG:fact1}
Let $f : {\mathcal X} \to {\mathcal S}$ be a map of 
finite type, \thref{formal:rev1}, between formal champs in
in the smooth topology 
with $\mathcal{X}$ consistent in a Japanese way, \thref{con:defn1}, 
then for any coherent sheaf $\mathscr{F}$ on $\mathcal{X}$ there is a coherent sheaf of ideals $\Theta_f (\mathscr{F})$ such that for any commutative, in which the square is fibred, diagram:
\begin{equation}
\label{flatG.12}
\xymatrix{
\mathcal{X} \ar[d]_f &\mathcal{X}_S \ar[l]^{\sigma} \ar[d]^{\varphi} &X \ar[l]^{\xi}\\
\mathcal{S} &S \ar[l]^{\sigma}
}
\end{equation}
with $\xi , \sigma$ flat, and $X$, $S$ disjoint unions of affines,
\begin{equation}
\label{flatG.412}
\sqrt{\sigma^* \Theta_f (\mathscr{F})}
=\Theta_\phi (\sigma^*\mathscr{F}),\quad
\sqrt{\xi^*\sigma^* \Theta_f (\mathscr{F})}
=\Theta_{\phi\xi} (\xi^*\sigma^*\mathscr{F}).
\end{equation}
In particular, therefore, \thref{flatG:rule1} and \thref{flatG:rule3} hold for arbitrary
diagrams of champ.
\end{fact}
\begin{proof}
Starting from \eqref{flatG.412} the in particular is a straightforward diagram chase;
op. cit. itself, however is a litlle fastidious since a composition of non-flat can
be flat. As such, consider first the case of \eqref{flatG.412} where $\mathcal{X}$
is a disjoint union $U/S$ of formal affinoids, and $\mathcal{S}$ the classifier of a groupoid 
$(s,t):R\rightrightarrows S$, then the fibre square in \eqref{flatG.12} becomes,
\begin{equation}
\label{flatG.413}
\xymatrix{
U \ar[d]_f & \,\, R_U:=R_{s}\times_S U \ar[l]^{\sigma} \ar[d]^{\varphi} \\
[S/R] &S \ar[l]^{\sigma}
}
\end{equation}
As such, the good candidate is $\Theta_\phi (\sigma^*\mathscr{F})$ which 
could be very different from $\Theta_{U/S}$. Consequently in the first 
instance we need to descend this to $U$, but this is clear from the
source and sink in, $R_U\times_U R_U\rightrightarrows R_U$ flat and
\thref{flatG:rule1}. After which, we need independence of the presentation
in \eqref{flatG.413}, so, say $S'\to S$ a smooth (flat would do if 
$\mathcal{S}$ were universally consistent in a Japanese way) refinement, then for $U'$
the fibre over $U$, we have a diagram of fibred squares,
\begin{equation}
\label{flatG.414}
\xymatrix{
R_U \ar[d] & \,\, R_{U'}:=R_{s}\times_S U' \ar[l]^{} \ar[d]^{} & R'_{U'}
:=R'_s\times_{S'} U' \ar[l]\ar[d]\\
R   \ar[d] & R_{S'} \ar[l]^{}\ar[d] &R':=R\times_{S\times S} {S'\times S'}\ar[l] \\
S &S' \ar[l]^{}
}
\end{equation}
in which all the horizontals are flat, and we need to compare $\Theta_{R_U/S}$
with $\Theta_{R'_{U'}/S'}$. However by \thref{flatG:rule3}, the former pulls
back to $\Theta_{R_{U'}/S'}$, which, in turn, pulls back to the latter by
\thref{flatG:rule1}.
Now, given this case, the well definedness of $\Theta_f$, for $f$ an
arbitrary map of champ, is immediate from \thref{flatG:rule1}, and 
any refinement of $U$ in \eqref{flatG.413}. Similarly, there's nothing
to be done to get the first part of \eqref{flatG.412} from 
the well definedness in \eqref{flatG.413}, while again, \thref{flatG:rule1}
gives the rest.
\end{proof}

Now while one needs that $U$ in \eqref{flatG.12} is consistent in a Japanese way to guarantee that $\Theta_f$ is a coherent sheaf of ideals, or equivalently, the closed complements to \eqref{flatG.1} admit, globally, a reduced formal scheme structure, this is automatic in the discrete, a.k.a. scheme, case so,

\begin{rmk}\thlabel{flatG:rmkE1}
Even when $\mathcal{X}$ of \eqref{flatG.12} is not consistent in a Japanese way, \thref{con:defn1}, it still follows,in the notation of \eqref{flatG:bonus1}, that we have an identity of closed subsets
of the trace,
\begin{equation}
\label{flatG.E1}
s^{-1} Z_f ({\mathscr F} \vert_U) = t^{-1} Z_f ({\mathscr F} \vert_U) = Z_f ({\mathscr F} \vert_R)
\end{equation}
and this, i.e. a well defined trace, is, all we need for \thref{flatG:def2} to have sense, albeit when we come to flattening by blowing up, \S \ref{S:end}, we will need consistency hypothesis, but this is on the image of \eqref{flatG.E1} in $S$, rather than on ${\mathcal X}$ itself.
\end{rmk}

In any case, for convenience let us recall, \cite{mur}[\S3. Thm 2], \cite{ray}[4.1.2],

\begin{revision}\thlabel{flatG:rev1}
Let $f : \mathcal{X} \to S = {\rm Spf} (A)$ be a formal champ in the smooth topology over a Noetherian affinoid base and $\mathscr{F}$ a coherent sheaf on $\mathcal{X}$ then on the category, ${\rm Aff}/S$, of affine formal $S$-schemes, the flatifier, $F_{\mathscr{F}}$, is the functor,
\begin{equation}
\label{flatG.15}
F_{\mathscr{F}} (T) = \left\{\begin{matrix}
\{\emptyset \} &\text{if} \ \mathscr{F}_T/T \ \text{is flat}, \hfill \\
\emptyset &\text{otherwise.} \hfill
\end{matrix}\right.
\end{equation}
\end{revision}

Closely related to this is,

\begin{defn}\thlabel{flatG:def2}
Let everything be as in \thref{flatG:rev1} and denote by $\vert Z \vert_f(\mathscr{F})$ the points $x$ of $\mathcal{X}$ where $\mathscr{F}_x$ fails to be a flat $A$-module then we say that {\it non-flatness of $\mathscr{F}$ is pure} 
if $\vert Z \vert_f (\mathscr{F})$, resp. $f(\vert Z \vert_f (\mathscr{F}))$,
are closed in the trace of $\mathcal{X}$, resp. $S$.
Similarly,
we say that {\it non-flatness of $\mathscr{F}$ is universally pure} if non-flatness of every base change of $\mathscr{F}$ in ${\rm Aff}/S$ is pure.
\end{defn} 

As one might  imagine this is exactly what is required to establish,

\begin{fact}\thlabel{flatG:fact2}
Let everything be as in \thref{flatG:rev1}, and suppose further that $f$ is quasi-compact of finite type, 
\thref{formal:rev1}, then $F_{\mathscr{F}}$ is representable by a formal $S$-scheme if the non-flatness of $\mathscr{F}$ is universally pure, while, conversely, universal purity in
the sense of \eqref{flatG:def2}, is necessary for representability, and, irrespective of the hypothesis on $f$, should the functor be representable then $F_{\mathscr{F}} \to S$ is of finite type and everywhere net.
\end{fact}

\begin{proof}
In the first instance consider the discrete case, i.e. $S$ a scheme,
then, \cite{mur}[Theorem 2], we begin with the equivalence of the
purity condition of \thref{flatG:def2} with the algebrisation axiom,
(F3) of op. cit., i.e. \eqref{flatG.17} below. Specifically, if (F3) holds then
$F_{\mathscr{F}} \to S$ is representable. 
If, however,
$f$ is of finite type, and the non-flatness of $\mathscr{F}$ is not universally pure, 
then, after an implicit base change, by \thref{flatG:rmkE1},  $\vert Z \vert_f (\mathscr{F})$ is closed 
but the constructible set $f(\vert Z \vert_f (\mathscr{F}))$
is not.
In particular if $s \in \overline{f(\vert Z \vert_f (\mathscr{F}))}\backslash 
f(\vert Z \vert_f (\mathscr{F}))$, 
then, the completion, $\widehat{\mathscr{F}}$ on $\widehat{\mathcal{X}}_s$, of $\mathcal{X}_s$ in the $\mathfrak{m}(s)$-adic topology is flat over the completion $\widehat A$ in the same. Consequently, there is a map,
\begin{equation}
\label{flatG.16}
{\rm Spec} \, \widehat A \longrightarrow F_{\mathscr{F}}
\end{equation}
while by items (1) and (2) of \thref{flatG:bonus1}, $\mathscr{F}$ isn't flat at every prime of $\widehat A$ lying over $f(\vert Z \vert_f (\mathscr{F}))$, so the existence of the map \eqref{flatG.16} is false.

Suppose conversely that the non-flatness of $\mathscr{F}$ is universally pure,
then we require an isomorphism,
\begin{equation}
\label{flatG.17}
F_{\mathscr{F}} (\varprojlim R/\mathfrak{m}^n) \xrightarrow{ \ \sim \ } \varprojlim F_{\mathscr{F}} (R/\mathfrak{m}^n)
\end{equation}
To this end, base changing as necessary, we may suppose that $A = \widehat A$ is also $\mathfrak{m} (s)$ complete in the notation immediately preceding \eqref{flatG.16}, while the right hand side of \eqref{flatG.17} is non-empty, so, by \thref{flatG:bonus1}, $\mathcal{X}_s \cap \vert Z \vert_f (\mathscr{F})$ is empty. As such, by \thref{flatG:def2} for each of the finitely many generic points $\zeta$ of $\vert Z \vert_f (\mathscr{F})$, $s \notin \overline{f(\zeta)}$, and $\widehat A = A$ is a local ring so $\mathscr{F}$ is flat.

The passage from the discrete case to the formal case is straightforward. 
Specifically for each $n\geq 0$, let $S_n$ be the spectrum of $A/I^{n+1}$,
$\mathscr{F}_n:=\mathscr{F}\mid_{S_n}$,
then we have flatifiers,
\begin{equation}
\label{flatG.101}
F_n:=F_{\mathscr{F}_n}\to S_n,\quad F_n=F_{n+1}\times_{S_{n+1}} S_n
\end{equation}
and, of course, $\mathscr{F}_T$ is flat iff $\mathscr{F}_{T_n}$ is flat for all $n$,
so the formal champ,
\begin{equation}
\label{flatG.102}
\varinjlim_{n} F_n\longrightarrow S
\end{equation}
certainly represents $F_{\mathscr{F}}$. Better still, by \eqref{flatG.101},
it satisfies the conditions of \thref{fact:quasi4}, so it's of
finite type, and it's net because any finite type map is net as
soon as its trace is, \thref{fact:quasi3}.
\end{proof}

By way of clarity it's therefore appropriate to note, 

\begin{rmk}\thlabel{flatG:rmk1}
Plainly if $f : \mathcal{X} \to S$ is proper then, by \thref{flatG:cor1}, the non-flatness of $\mathscr{F}$ is universally pure. Slightly less obviously, \thref{A:formalpureT}, if $\mathscr{F}$ is formally universally pure then its non-flatness is universally pure. The converse is, however, manifestly false, e.g. just remove a fibre of something flat. 
\end{rmk}

In any case, in the discrete case, the purity of non-flatness is the same thing as 
$f(Z_f(\mathscr{F}))$ closed, but exactly what this means in the general case without
some properness assumption is far from clear. As such we make,

\begin{defn}\thlabel{flatG:defn100}
Let everything be as in \thref{flatG:fact1}, with $Z_f(\mathscr{F})\hookrightarrow\mathcal X$
the closed sub-champ defined by $\Theta_f(\mathscr{F})$ of op. cit., then for 
$Z_f(\mathscr{F})/S$ proper we deonote by $\mathfrak{Z}(\mathscr{F})\hookrightarrow S$,
or just $\mathfrak{Z}$ if there is no confusion, the closed formal subscheme defined
by the kernel of,
\begin{equation}\label{flatG.103}
\mathscr{O}_S\rightarrow f_* \mathscr{O}_{Z_f(\mathscr{F})}
\end{equation}
\end{defn}

Irrespectively, to draw global consequences from \thref{flatG:def2}, it suffices to observe,

\begin{rmk}\thlabel{flatG:rmk2}
By its very definition, \thref{flatG:rev1}, if $S' \to S$ is in ${\rm Aff}/S$ and $\mathscr{F}'$ the base change of $\mathscr{F}$ to $S'$ then,
\begin{equation}
\label{flatG.19}
F_{\mathscr{F}'} = (F_{\mathscr{F}}) \times_S S' \, .
\end{equation}
Better still if there is a map to $F_{\mathscr{F}}$ it is, by \eqref{flatG.15}, unique.
\end{rmk}

Applying this we therefore obtain,

\begin{cor}\thlabel{flatG:cor2}
Let everything be as in \thref{flatG:fact1} with $f$ quasi compact then there is a finite type, \thref{formal:rev1},  representable net map $F_{\mathscr{F}} \to {\mathcal S}$ representing the $1$-functor on champs over ${\mathcal S}$ given by,
\begin{equation}
\label{flatG.20}
T \mapsto \left\{\begin{matrix}
\{\emptyset \} &\text{if} \ \mathscr{F}_{T}/T \ \text{is flat}, \hfill \\
\emptyset &\text{otherwise} \hfill
\end{matrix}\right.
\end{equation}
if $f$ is proper, or more generally $\mathscr{F}$ is pure, or, more generally still the closed subsets of the trace defined by the ideals $\Theta_f (\mathscr{F})$ of \thref{flatG:def1} are universally pure.
\end{cor}

\begin{proof}
By \thref{flatG:rmk2} the functors $F_{\mathscr{F}}$ of \thref{flatG:rev1} globalise to a sheaf on ${\mathcal S}$, so $F_{\mathscr{F}}$ is representable by a representable map iff it is so everywhere locally. The necessary and sufficient condition for this, \thref{flatG:fact2}, that the non-flatness of $\mathscr{F}$ is universally pure simplifies, in the presence of the existence of the ideals $\Theta_f$, to the universal purity of an irreducible component of their support.
\end{proof}




%% file: 2Pointed_Flatifier.tex
\section{The pointed flatifier}\label{PointFlatifier}

The essential local description of the flatifier has been provided by Hironaka in which the starting point is a trivial variant of \cite{hironaka}[Thm 2.1], to wit:

\begin{lem}\thlabel{flatP:lem1}
Let $A$ be a complete local Noetherian ring with maximal ideal ${\mathfrak m}$ and infinite residue field $k=k({\mathfrak m})$ with $M$ a finitely generated module over $B = A [[T_1 , \cdots , T_n]]$ then, possibly after an $A$-linear change of coordinates $T_i$, if $B_m = A [[T_1 , \cdots , T_m]]$, $1 \leq m \leq n$, there are $B_m$ linear maps $\theta_{m\ell} : B_m \to M$, $1 \leq \ell \leq \nu_m$ such that,
\begin{equation}
\label{flatP.1}
\theta := \coprod_{1 \leq m \leq n} \ \coprod_{1 \leq \ell \leq \nu_m} \theta_{m\ell} : \coprod_{1 \leq m \leq n} B_m^{\oplus \nu_m} \to M
\end{equation}
is surjective with kernel contained in the $A$-module generated by ${\mathfrak m}$.
\end{lem}

\begin{proof}
The case $n=0$ is Nakayama's lemma, and we proceed by induction on the same. Similarly, in a minor variant of Nakayama, we may suppose that we have a presentation,
\begin{equation}
\label{flatP.2}
0 \to K \to E := B^{\oplus e} \to M \to 0
\end{equation}
which doesn't factor through any vector bundle quotient of $E/{\mathfrak m}$ modulo ${\mathfrak m}$. Equally plainly, for any vector bundle quotient $E \twoheadrightarrow E''$ we get a diagram,
\begin{equation}
\label{flatP.3}
\xymatrix{
&0 \ar[d] &0 \ar[d] &0 \ar[d] \\
0 \ar[r] &K' \ar[r] \ar[d] &E' \ar[r] \ar[d] &M' \ar[r] \ar[d] &0 \\
0 \ar[r] &K \ar[r] \ar[d] &E \ar[r] \ar[d] &M \ar[r] \ar[d] &0 \\
0 \ar[r] &K'' \ar[r] \ar[d] &E'' \ar[r] \ar[d] &M'' \ar[r] \ar[d] &0 \\
&0 &0 &0
}
\end{equation}
In particular, therefore, if we take $E''$ to be a quotient of maximal rank such that $K''$ is zero in $M''$ modulo ${\mathfrak m}$, it will suffice to prove \thref{flatP:lem1} for $M'$. Similarly if there is a sub-bundle $E'$ such that $K'$ is zero $\mod {\mathfrak m}$ it will suffice to prove op. cit. for $M''$. Thus, without loss of generality, we may suppose that for every diagram of the form \eqref{flatP.3} everything in the left hand column has non-zero image in the middle column mod ${\mathfrak m}$. 

As such by induction on the rank $e$ of $E$, without loss of generality, there is a function $f \in B$ which is non-zero modulo ${\mathfrak m}$ such that the support of $M$ is contained in the hypersurface $f=0$.  Indeed the case $e=1$ is the definition of $K$, \eqref{flatP.3}, not contained in ${\mathfrak m} E$, and the induction is automatic as soon as $K',K''$ in op. cit. have the same property. Now let $s$ be the order of such a $f$ modulo ${\mathfrak m}$, and consider the leading term,
\begin{equation}
\label{flatP+1}
f_s := \sum_{\vert I \vert = s} a_I \, T^I \in k [T_1 , \cdots , T_n]
\end{equation}
then we have  a wholly algebraic situation, and, since $k$ is infinite, a generic $k$-projection of $f_s = 0$ is finite over its image, i.e. after a generic change of coordinates the coefficient of $T_n^s$ in \eqref{flatP+1} is non-zero. As such the preparation theorem of \cite{bourbaki}[VII.3.8, Prop. 6] applies to conclude that,
\begin{equation}
\label{flatP+2}
B_{n-1} \longrightarrow B_n = B/f
\end{equation}
is finite, so $M$ is finite over $B_{n-1}$ and we conclude by induction on $n$.
\end{proof}

The relevance of this is that the ``coefficient ideal'' of \cite{hironaka}[Thm 1.14] which will define the local flatifier is non-trivial, i.e.

\begin{factdef}\thlabel{flatP:defn1}
In the notation of \thref{flatP:lem1} let $E_m$ be the free $B_m$ module $B_m^{\oplus \nu_m}$ then the coefficient ideal, $J_{\theta}$, of the kernel of \eqref{flatP.1} is the ideal generated by,
\begin{equation}
\label{flatP.4}
\left( D(e)(0) \mid D = \sum_{m=1}^n D_m \in {\rm Diff}_{B_m/A} (E_m , B_m) , e \in {\rm Ker} (\theta) \right).
\end{equation}
In particular, therefore, in the presence of the conclusion of \thref{flatP:lem1}, $J_{\theta} \subseteq {\mathfrak m}$.
\end{factdef}

As promised we can employ this to describe,

\begin{defrev}\thlabel{flatP:defn2}
Let everything be as in \thref{flatG:rev1} with $s$ a point of $S$ and ${\rm Aff}^* / (S,s)$ the category of pointed affinoids over $S$, then $F_{\mathscr F}^*$ is the set valued functor,
\begin{equation}
\label{flatP.5}
F_{\mathscr F}^* (T,t) = \left\{\begin{matrix}
\{\emptyset \} &\mbox{if} \ \mathscr{F}_T/T \ \mbox{is flat}, \hfill \\
\emptyset &\mbox{otherwise.} \hfill
\end{matrix}\right.
\end{equation}
\end{defrev}

Of which the precise relation to the coefficient ideal is,

\begin{fact}\thlabel{flatP:fact2}
cf. \cite{hironaka}[Thm 1.14] Let everything be as in \thref{flatG:fact2} and \thref{flatP:defn2}, and suppose moreover that $S$ is strictly Henselian then $F_{\mathscr F}^*$ is represented by a non-empty closed sub-scheme $F_{\mathscr F}^* \hookrightarrow S$ pointed in $s$. Better still if ${\mathcal X}$ of \thref{flatG:fact2} is the spectrum of a strictly Henselian local ring smooth, \thref{con:cor1}, over $S$, then for $M = \Gamma ({\mathscr F})$ and any presentation satisfying \thref{flatP:defn1} the completion of $F_{\mathscr F}^*$ in $s$ is the sub-scheme defined by the coefficient ideal $J_{\theta}$ of \eqref{flatP.4}.
\end{fact}

\begin{proof}
Since everything is flat over a field, the fibre of the flatifier $F_{\mathscr F}$ of \thref{flatG:fact2} over $s$ is non-empty, and, around each point $*_j$ of the fibre $F_{\mathscr F}$ is \'etale locally isomorphic to a closed sub-scheme $F_j$ of $A$, so, plainly $F_{\mathscr F}^*$ is represented by
\begin{equation}
\label{flatP.6}
\coprod_j (F_j , *_j) \to (S,s) \, .
\end{equation}

Over any ring, however, a diagram chase reveals that if a module is flat modulo ideals $I$ and $J$ then it's flat modulo $I \cap J$, so, by definition the left hand side of \eqref{flatP.6} is also the closed sub-scheme defined by the intersection of the ideals of the $F_j$, i.e. there was only ever one $*_j$ to begin with.

As to the better still. Plainly, by definition, the completion $\widehat F_{\mathscr F}^*$ of the pointed flatifier is the flatifier, $F_{\widehat{\mathscr F}}^*$, of the completion whose ideal $J$ in the complete local ring of ${\mathcal X}$, certainly contains $J_{\theta}$ since each $B_m$ in \eqref{flatP.1} is flat over $A$. The converse, is more difficult, but is exactly \cite{hironaka}[Lemma 1.14.5].
\end{proof}

This brings us to the key point of the discussion,

\begin{prop}\thlabel{flatP:prop1} \emph{\cite{hironaka}[Theorem 2.4, item 3]}
Again for everything  as in \thref{flatG:fact2} and \thref{flatP:defn2} with $S$ strictly Henselian; ${\mathcal X}/S$ of finite type, \thref{formal:rev1}, and suppose that the pointed flatifier $F_{\mathscr F}^* \hookrightarrow S$ is a Carter divisor then the composition,
\begin{equation}
\label{flatP.7}
{\rm H}^1_{f^{-1} F_{\mathscr F}^*} ({\mathcal X} , {\mathscr F}) \hookrightarrow {\mathscr F} \twoheadrightarrow {\mathscr F} \vert_{{\mathcal X}_s}
\end{equation}
is not zero at some point of the fibre ${\mathcal X}_s$.
\end{prop}

\begin{proof}
By \thref{flatP:fact2}, for every point $x$ of ${\mathcal X}_s$, there is an ideal $I_x$ of ${\mathscr O}_{S,s}$ defining the pointed flatifier $F_{{\mathscr F}_x}^*$ as a closed sub-scheme of $S$, and plainly,
\begin{equation}
\label{flatP.8}
F_{\mathscr F}^* \hookrightarrow \underset{x \in {\mathcal X}_s}{\cap} F_{{\mathscr F}_x}^* \, .
\end{equation}

Conversely, however, the sub-scheme on the right of \eqref{flatP.8} defines the flatifier after completion in ${\mathfrak m} (s)$, so, just as in \eqref{flatG.16} et seq., the purity condition in \thref{flatG:def2} ensure that it's the pointed flatifier over $S$. As such the Cartier hypothesis of \thref{flatP:prop1} amounts to the existence of $\gamma$ in ${\mathscr O}_{S,s}$ such that,
\begin{equation}
\label{flatP.9}
{\mathfrak m} (s) \supseteq (\gamma) = (I_x \mid x \in {\mathcal X}_s)
\end{equation}
and whence,  by \thref{flatP:fact2}, for some point $x$, the coefficient ideal $J_{\theta} = \widehat{I}_x$, for an appropriate presentation of $\widehat{\mathscr F}_x$ as a module over the complete local ring $\widehat{\mathscr O}_{X,x}$, is Cartier. Thus following Hironaka, \cite{hironaka}[1.14.3], albeit in the notation of \eqref{flatP.3} et seq., there are elements $e,\varepsilon$ of the presentation \eqref{flatP.1} such that $e = \gamma \varepsilon$, with $\theta(e) = 0$, but $\theta(\varepsilon) \notin {\mathfrak m} (s) {\mathscr F}_x$, so the element $\theta(\varepsilon)$ in the left most group of \eqref{flatP.7} is non-zero in the rightmost group.
\end{proof}

Before continuing this discussion we may usefully make,

\begin{warning}\thlabel{flatP:warn1}
Already for algebraic spaces, even under the hypothesis of \thref{flatG:fact2} and $S$ strictly Henselian, the pointed flatifier of \thref{flatP:defn2} may very well fail to be the flatifier of \thref{flatG:rev1}  because the latter has components which don't contain $s$. An example is provided in \cite{hironaka}[Example 2].
\end{warning}

In any case, the principle
further complication, as David Rydh pointed 
out to me, is dense components of the
pointed flatifier of \thref{flatP:fact2}, and whence we need:

\begin{factdef}\thlabel{flatP:fact4}
Let everything be as in \thref{flatG:fact1} \& \thref{flatG:fact2}, 
with $Z_f(\mathscr{F})/S$ proper,
then there is a unique closed subspace $E_{\mathscr F}$ of $F_{\mathscr F}$,
{\it the fitted flatifier},
such that in the (strictly Henselian) local ring $\mathscr{O}_{F,s}$
the ideal of $E_{\mathscr F}$ is the pull-back 
of the ideal of $V'$, 
where in any sufficiently small \'etale neighbourhood $V\to S$ of $s$,
$V_{\mathrm{red}}=V'\cup V''$ is the
decomposition into components (whether of $V$ or the local ring) 
over which ${\mathscr F}$ is not, resp. is, flat.
In particular the reduction of  $E_{\mathscr F}$ factors through
$\mathfrak{Z}({\mathscr F})$
of \thref{flatG:defn100}, while respecting smooth base change $S'\rightarrow S$, i.e.
\begin{equation}\label{flatP.101}
E_{\mathscr F}\mid_{S'}=E_{\mathscr F_{S'}}
\end{equation}
\end{factdef}
\begin{proof}
Plainly if $E_{\mathscr{F}}$ is well defined, then its definition
respects smooth base change. Similarly,
if $\mathfrak{Z}({\mathscr F})$ is empty, then $E_{\mathscr F}$ is empty, and
there is nothing to do, while, otherwise,
we proceed by way of induction on the dimension of the trace, $\vert Z\vert$, of, 
$\mathfrak{Z}:=
\mathfrak{Z}({\mathscr F})$, so,
it's closed by hypothesis.
Before beginning, however,
observe, independently of the dimension of $\vert Z\vert$, 
there is a dichotomy,
\begin{claim}\label{flatP:claim11}
Let everything be as above then the reduction of the pointed flatifier
either factors through $\mathfrak{Z}$, or contains a component of $V$. 
\end{claim}
\begin{proof}[proof of \thref{flatP:claim11}]
Indeed if the former
is false then in, say, $B={\mathscr O}_{S_{\mathrm{red}},s}$, its ideal, $J$, is supported
at a prime $\mathfrak{p}$ which isn't in the support of $\mathfrak{Z}$. As such, there is a function
$g\in \mathscr{O}_{S,s}$ whose pull-back vanishes on $Z_f(\mathscr{F})$, but
$g(\mathfrak{p})\neq 0$. Consequently, by \thref{flatG:fact1}, for any $\mathcal{X}\ni x\mapsto s$, in the
local, \thref{hensel:factdef1} sense, ring $\mathscr{O}_{\mathcal{X},x}$, $(\mathscr{F}_x)_g$ is flat
over $B_g$, so, a fortiori
$({\mathscr F}_x)_{\mathfrak{p}}$ is flat at $B_{\mathfrak{p}}$.
Now we'd like to apply \cite{ega4}[Th\'eor\`eme 11.2.6] to the directed limit of
\eqref{hensel2}, but, a priori, op. cit. supposes finite type in the
discrete sense rather than the more general sense of \thref{formal:rev1}.
However, this is only used to ensure that non-flatness is a closed condition,
and we have this from \thref{flatG:fact1}, so, for a suitably small \'etale
neighbourhood $V\to S$ of $s$, there is a cover $U\to \mathcal{X}$, such
that $\Gamma(U,\mathcal{F}_U)_{\mathfrak{p}}$ is flat over $\Gamma(V)_\mathfrak{p}$,
and whence 
by \cite{ega4}[Th\'eor\`eme 11.4.1], ${\mathscr F}$ is flat along.
\begin{equation}\label{flatP.104}
B/Q:=\cap_n \mathfrak{p}^{(n)},\,\text{i.e. n{\it th} symbolic power},
\end{equation}
so $Q\supseteq J$, while $Q$ is zero in $B_\mathfrak{p}$, from which the
annihilator of $J$ is non-zero, and whence, \eqref{rank55}, $J$ is supported
on a component. 
\end{proof}

In particular, if the definition of $E_{\mathscr F}$ is non-empty at $s\in \vert Z\vert$,
then from $Z_f({\mathscr F})=Z_f({\mathscr F\mid_{V'}})$, its fibre through $s$
is the pointed flatifier of ${\mathscr F\mid_{V'}}$ which, in turn, is a
thickening of  that of $\mathscr{F}\mid_{\mathfrak{Z}}$. 
Turning to the induction, the exact statement will be that for $d\geq 0$,
there is an open $\Sigma\hookrightarrow S$ such that the co-dimension of
$S\backslash\Sigma$ intersected with the trace of $\mathfrak{Z}$ is at least $d$,
and \thref{flatP:fact4} holds over $\Sigma$. We can, therefore, start
from $d=0$ with $\Sigma$ empty, and we fix $\zeta$ in the trace of $\mathfrak{Z}$
of co-dimension $d\geq 0$. As such by \thref{flatG:fact2} 
we may suppose that the ideal of the pointed flatifier at $\zeta$ is defined
in an \'etale neighbourhood $V\rightarrow S$ of
the same, whose components are equally those of the Henselian local ring,
and we can define $V''$ to be any components of $V$ on which it is supported
with $V'$ the rest. Consequently, if we gain again denote the ideal of
the pointed flatifier of
${\mathscr F}_{V}$ at $\zeta$ by $J$, with $I'$ the ideal
of $V'$, then $J$ mod $I'$ has no support on minimal primes
of $V'$,
so, cf. \eqref{rank55}, this is still true in any finer \'etale
neighbourhood of $V'$. At the same time, if for ease of notation, we suppose
$V\rightarrow S$ is onto, with $S^\times=S\backslash \{\zeta\}$ then, 
by the induction, $E_{\mathscr{F}}\mid_{S^\times}$
is a well defined closed subspace of $F_{\mathscr{F}}\mid_{S^\times}$, 
with ideal $K$. Furthermore, 
for $V$ sufficiently small, with   $V^\times=V\backslash\{\zeta\}$,
the decomposition of $F_{\mathcal{F}}$ into $\overline{F}\hookrightarrow V$ and $F^\times\to V^\times $ of 
\thref{fact:quasi2} affords  the
following cases,

\begin{enumerate}
\item[(0)] Over $F^\times$, the definition of \thref{flatP:fact4} implies that
we do nothing, and $K$ continues to define $E_{\mathscr{F}}\hookrightarrow
F_{\mathscr{F}}$.

\item[(1)] At any points of $\overline{F}\backslash V'$, 
 $\mathscr{F}\mid_{V_{\mathrm{red}}}$ is flat,
so both $K$ and the ideal of $I'$ are the whole ring.

\item[(2)] Points of the fibre $\overline{F}_{V'}$, may be identified with the support 
of the ideal
$J+I'$ of the pointed flatifier of $\mathscr{F}_{V'}$, which, as we've said, is 
nowhere schematically dense in $V'$, and whence, again
$K$ coincides with the ideal of $V'$ restricted to $V^\times$.
\end{enumerate}

\noindent Moreover, in order to glue from $V$ to its image in $S$,
which for ease of notation we confuse with $S$ itself,
either,

\noindent (a) The restriction to  $V^\times$   
of \thref{flatP:fact4}  defines
a non-empty sub-scheme of the pointed flatifier of $\mathscr{F}_V$,
and is equally the pointed flatifier of $\mathscr{F}_{V'}$.
Consequently, 
by
\cite{ega4}[Th\'eor\`eme 11.4.1], for $j:S^\times\hookrightarrow S$
the inclusion, we have a fibre square,
\begin{equation}\label{flatP.102}
\begin{CD}
(j_* K)\mid_{V} @>>> (j_*\mathscr{O}_{F_{\mathscr{F}}})\mid_V \\
@AAA @AAA \\
I'  \mathscr{O}_{F_\mathscr{F}}\mid_V @>>> \mathscr{O}_{\overline{F}_\mathscr{F}}\mid_V,
\end{CD}
\end{equation}
where $\widetilde{F}_\mathscr{F} \hookrightarrow F_\mathscr{F}$
is the closed subspace of \thref{rmk:jap1} no associated prime
of which factors through $S\backslash\Sigma$. In particular,
(the not necessarily coherent in the non-discrete
case) ideal $(j_* K)\cap  \mathscr{O}_{F_\mathscr{F}}$ 
is, in fact, coherent, and defines a closed subspace,
\begin{equation}\label{sat0}
\widetilde{E}_{\mathscr{F}}\hookrightarrow \widetilde{F}_{\mathscr{F}}
\end{equation}
over all of $S$.
Equally, \thref{flatP:fact4} defines $E_{\mathscr{F}}$ independently of
any Galois action at the separable point, so for $\hat{\bullet}$ completion
at $\zeta$, we have a well defined sub-scheme,
\begin{equation}\label{sat1}
\widehat{E}_{\mathscr{F}}\hookrightarrow \widehat{F}_{\mathscr{F}}
\end{equation}
which we can glue to \eqref{sat0} by way of,
\begin{equation}\label{sat2}
\widehat{\widetilde{E}}_{\mathscr{F}}\rightrightarrows \widehat{E}_{\mathscr{F}}\coprod\widetilde{E}_{\mathscr{F}}
\twoheadrightarrow {E}_{\mathscr{F}}
\end{equation}
to complete the verification that $E_{\mathscr{F}}\hookrightarrow
F_{\mathscr{F}}$ is well defined.

\noindent (b) In the discrete case the alternative that
$J':=J+I'$ has empty support in $V^\times$ is equivalent it to its radical
being $\mathfrak{m}(\zeta)$, so it is immediate that $J'$
is already defined in $S$, and, indeed, \thref{flatP:rmk1}, more
is true, but in general,
we need to exhibit a descent datum for $J'$ along $(s,t):R=V\times_S V
\rightrightarrows S$. By hypothesis, however, if the support of $J'$ is 
non-empty at $\rho\in R$, then $s(\rho)=t(\rho)=\zeta$, so $s^*J'$, resp. $t^*J'$,
is the pointed flatifier of $\mathscr{F}\mid_{s^{-1}(V')}$, resp.
$\mathscr{F}\mid_{t^{-1}(V')}$, and since neither is schematically
dense they both coincide with the pointed flatifier 
$\mathscr{F}\mid_{R'}$ at $\rho$. As such $J'$ defines a closed
sub-scheme $T\hookrightarrow S$, with trace exactly $\zeta$, and 
the connected component of
$E_{\mathscr{F}}$ through $\zeta$ is, identically, the fibre of 
$F_{\mathscr{F}}\mid_T$.
\end{proof}

Since, plainly, the definition of $V'\hookrightarrow V$ in 
\thref{flatP:fact4} isn't global this merits,

\begin{rmk}\label{flatP:rmk1}
In the concluding case (b), even $s^{-1}(V')$ and $t^{-1}(V')$
coincide in $\mathscr{O}_{R,\rho}$ with $R'$. Nevertheless 
this doesn't constitute a descent datum for the closed embedding
$V'\hookrightarrow V$ since, for example. $s^{-1}(\zeta)$ can
perfectly well be empty where $s^{-1}(V')$ is non-empty. More 
explicitly in the discrete case, there is a strictly decreasing,
i.e. no common generic points,
filtration,
\begin{equation}\label{flatP103}
Z^0=V',\,\, Z^{i+1}=\mathfrak{Z}(\mathscr{F}_{Z^i}),\,\, i\geq 0
\end{equation}
and case (b) only occurs if $\zeta$ is a generic point of
some $Z^{i}$, $i\geq 1$, whose pointed flatifier is supported
in the same. Indeed if $i$ were maximal such that $Z^i\ni\zeta$,
then its pointed flatifier contains at least $Z^i$, so it must
be generic to be in case (b), while, equally, it's pointed flatifier 
has no support off $\zeta$, so, in the notation of op. cit.
the radical of $J'=J+I'$ is $\mathfrak{m}(\zeta)$, and it's
sufficient that $J'$ is a priori
well defined after strict Henselisation, or, indeed completion.
\end{rmk}

Irrespectively, in a critical respect the fitted flatifier mimics
the flatifier, i.e.

\begin{claim}\label{flatP:fact5}
In a sufficiently small \'etale neighbourhood $V\to S$, 
of $s\in S$, identify the
(implicitly non-empty) piece,
$E^*_{\mathscr{F}}$ 
of the fitted flatifier of \thref{flatP:fact4}
through $s$ 
with a closed sub-scheme, while denoting by
$\rho:W\to V$ the blow up in the same, $D$ the
exceptional divisor, and $W'$ the proper transform
of $V'$ of op. cit., then, at every $w\in W'\cap D$,
the pointed flatifier of $\mathscr{F}_W$ is $D$,
and whence
the conclusion of
\thref{flatP:prop1} holds, i.e. the composition,
\begin{equation}
\label{flatP.107}
{\rm H}^1_{D} ({\mathcal X}_W , {\mathscr F}_W) \hookrightarrow {\mathscr F}_W \twoheadrightarrow {\mathscr F} \vert_{{\mathcal X}_w}
\end{equation}
is not zero at some point of the fibre ${\mathcal X}_w$.
\end{claim}
\begin{proof}
The further conclusion of \eqref{flatP.107} is clear, and otherwise, 
again,
for $I'$, resp. $J$, the ideal in $\mathscr{O}_{S,s}$ of $V'$, 
resp. the pointed flatifier,
$F^*_{\mathscr{F}}$, 
with $\rho^*: \mathscr{O}_{S,s} \to \mathscr{O}_{W,w}$, then,
by definition
$\rho^*(J +I')$ is defined by a single function
$\gamma$. Consequently, there are functions 
$x,y$, resp. $\xi, \eta$, on $S$ at s, resp. $W$ at $w$, such that,
\begin{equation}\label{flatP104}
\gamma=x+y,\, J\ni x, \, I'\ni y,\, \rho^*x =\gamma\xi,\, \rho^* y=\gamma^n\eta,
\end{equation}
where, given $y$, there is a maximal $n\geq 1$ such that $y\in (I'+J)^n$,
and $\eta$ vanishes in $W'$. In particular therefore, $\eta\in\mathfrak{m}(w)$,
and whence $\rho^*(J+I')=(\rho^*x)=\rho^*J$.
\end{proof}

En passant this discussion allows us to establish that both the 
flatifier and the fitted flatifier are not too far from algebraic,
i.e.

\begin{fact}\label{flatP:fact6}
Let everything be as in \thref{flatG:fact1} \& \thref{flatG:fact2}, 
with $Z_f(\mathscr{F})/\mathcal{S}$ proper, and $\mathcal{S}$ universally consistent in
a Japanese way,
then for every irreducible 
component, $F$, 
of $F_{\mathscr F}$, resp. $E_{\mathscr F}$,
there is a closed formal sub-champ $\mathcal{T}\hookrightarrow \mathcal{S}$ such that,
$F_{\mathrm{red}}\to \mathcal{T}_{\mathrm{red}}$ is {\it almost a birational modification},
i.e. conditions (1)-(2) of \thref{A:mod} holds, but (3) of op. cit.
might only hold at points $t$ which are schematically dense in some \'etale local
component.
\end{fact}
\begin{proof}
The conditions determine $\mathcal{T}$ uniquely, so, without loss
of generality we may suppose that $\mathcal{S}=S$ is affine, and,
similarly to the proof of \thref{flatP:fact4} we'll prove, inductively
in $d\geq 0$, that there is a Zariski open $\Sigma_d\hookrightarrow S$
with a co-dimension $d$ complement and $T_d\hookrightarrow S$ closed
such that \thref{flatP:fact6} holds for the pair $T_d\cap \Sigma_d\hookrightarrow S_d$.
As such, we can certainly start with $d=0$ and $\Sigma_d$ empty. 
Now, quite generally, although its intersection with the trace may
not decrease, \eqref{flatP103} still defines, albeit with $Z^0=S$, a strictly
decreasing chain of formal sub-schemes, so, we add a further
induction on the maximum $i\geq 0$, 
$Z^i\ni\zeta$,
calculated at points $\zeta$ of the
trace of $S\backslash \Sigma$ of co-dimension $d$.
Consequently, the start of this further induction is
$F\xrightarrow{\sim} S$ a priori in an \'etale neighbourhood,
since without loss of generality $X'$ is empty in applying
\eqref{quasi3},
but being an isomorphism evidently descends, so, 
a posteriori in a Zariski neighbourhood
$U\ni \zeta$. In particular, a fortiori $d=0$, so gluing to
$T_d$ is empty, and we take $\Sigma_1$, for $i=0$, to
be the union of such neighbourhoods $U$ with $T_1$ all the
irreducible components of $S$ that they meet.
As to the next step, by \thref{flatP:claim11}, either the pointed flatifier
at $\zeta$, is contained in $\mathfrak{Z}$, so $F$ itself factors
through $\mathfrak{Z}$ and we've decreased $i$, or, the pointed
flatifier contains an \'etale local component of $S$. Again, this must,
glue to the given $T_d\hookrightarrow \Sigma_d$, or it misses
$\Sigma_d$ altogether. In the latter case $S\backslash \Sigma_d$
is a union of components of $S$, by \eqref{rank55}, 
so, by  another application of the same and the irreducibility of $F$, 
the fibre over $\Sigma_d$ is empty, i.e. without loss of
generality $d=0$, and whence, \'etale locally, the pointed flatifier, 
$F_\zeta$ through $\zeta$ may, by \thref{fact:quasi2}, be identified
with the fibre $F_V$ over an \'etale neighbourhood. As such, it admits
a descent datum, so, again, it's actually a union of components of $S$.
Finally, therefore, we have the possibility that the
pointed flatifier glues non-vacuously to $T_d$, but then
\eqref{rank55} applies once more to conclude that $T_d$ is a union
of irreducible components containing the \'etale local
component through $\zeta$.
\end{proof}









%% file: 3Compactification_General.tex
\section{Compactification in general}\label{S:flatC}

By a compactification, locally, is to be understood,

\begin{defn}\thlabel{flatC:defn1}
For $S = {\rm Spf} (A)$ a formal Noetherian affinoid, and $E \to S$ a quasi-finite formal space of formally finite type (e.g. the flatifier) by a compactification of $E$ is to be understood a commutative diagram,
\begin{equation}
\label{flatC.1}
\xymatrix{
E \ \ar[rd] \ar@{^{(}->}[rr]_j &&C \ar[ld]^{\pi} \\
&S
}
\end{equation}
in which $\pi$ is finite and $j$ is a Zariski dense open immersion. In the event that a compactification exists we will say that $E$ is compactifiable.
\end{defn}

Should the topology on $A$ be discrete, i.e. ${\rm Spec} (A) = {\rm Spf} (A)$, then, of course, Zariski's main theorem tells us that compactifications exist, but, in general, it is a highly non-trivial condition, since, for example:

\begin{fact}\thlabel{flatC:fact1}
Let everything be as in \thref{flatC:defn1} and suppose that $E$ is compactifiable then it is algebraisable, and is even the $I$-adic (i.e. the topology on $A$) completion of a scheme.
\end{fact}

\begin{proof}
$C$ is algebraisable. Indeed it's the $I$-adic completion of the scheme ${\rm Spec} \, \Gamma (C,$ ${\mathscr O}_C)$, which is, of course, finite over $S$, while $E$ is quasi compact, so its covered by finitely many open affinoids of $C$, each of which is algebraisable.
\end{proof}

In the formal category we require criteria for the existence of a compactification, and to this end we have,

\begin{lem}\thlabel{flatC:lem1}
Let $S = {\rm Spf} (A)$ be a Noetherian affinoid and $E \xrightarrow{\varphi} S$ a quasi-finite map of formally finite type such that there exists a (not necessarily fibred) commutative diagram,
\begin{equation}
\label{flatC.9}
\xymatrix{
G \ \ar[d]_{\gamma} \ar@{^{(}->}[r]_{\kappa} &P \ar[d]^{\pi} \\
E \ar[r]_{\varphi} &S
}
\end{equation}
with proper vertical arrows, $\kappa$ a dense open embedding, and $\gamma$ {\it dominant}, i.e. $\gamma^*$ is injective, then $E$ is compactifiable.
\end{lem}

\begin{proof}
Consider the at worst quasi-coherent equaliser,  
\begin{equation}
\label{flatC.10}
0 \to {\mathscr C} \to \pi_* {\mathscr O}_P \prod \varphi_* {\mathscr O}_E \overset{\kappa^*}{\underset{\gamma^*}{\rightrightarrows}} \varphi_* \gamma_* {\mathscr O}_G = \pi_* \kappa_* {\mathscr O}_G \, .
\end{equation}
Then since $\gamma$ is dominant, the projection of ${\mathscr C}$ to $\pi_* {\mathscr O}_P$ is certainly injective, so by \cite{ega3}[Th\'eor\`eme 3.4.2], ${\mathscr C}$ is coherent, and the resulting formal scheme $C := {\rm Spf} \, {\mathscr C} \to S$ is finite over $S$. It remains, therefore, to prove that the map $E \to C$ resulting from the second projection,
\begin{equation}
\label{flatC.11}
{\mathscr C} \to \varphi_* {\mathscr O}_E
\end{equation}
is open, which, we prove, following \cite{ega4}[Th\'eor\`eme 8.12.6], albeit
via the slightly different inductive statement on $d\geq 0$: there is an open embedding $\Sigma\rightarrow S$ with complement
of co-dimension $d$ in the trace of $S$ such that the map of fibres, $E_\Sigma\to C_\Sigma$, is open. As such in the inductive hypothesis we should suppose that $S$ is a formal scheme rather than just an affinoid, but otherwise everything including the diagram \eqref{flatC.9}, whose  defining characteristics respect arbitrary base change, is as above. Plainly, in co-dimension $0$ we can take $\Sigma$ empty,  
while to prove the statement for $d+1$, we take $s$ to be one of the
finitely many co-dimension $d$ points of the complement of $\Sigma$.
Now, by \thref{fact:quasi2}, there is an \'etale neighbourhood $V\to S$
of $s$, together with a decomposition,
\begin{equation}\label{flatC122}
E\mid_{V}= \overline{\Phi}\coprod \Phi'
\end{equation}
wherein $\overline{\Phi}\rightarrow V$ is finite, and $\Phi'$
factors through the fibre over $\Sigma$. 
In particular the fibre $G_{\bar{\Phi}}$ is proper over $S$, so 
$\kappa : G_{\bar{\Phi}} \to P$ is both open and closed. Consequently the functions on $G$ defined by the projection,
\begin{equation}
\label{flatC.13}
{\mathscr O}_{\overline{\Phi}}  \underset{\gamma^*}{\longrightarrow} \varphi_* \gamma_* {\mathscr O}_{G}
\end{equation}
are the image of functions on $P$, and the natural map,
\begin{equation}
\label{flatC.14}
{\mathscr C} \to {\mathscr O}_{\Phi}
\end{equation}
admits a retraction, i.e. $\overline{\Phi}$ is a connected component of $C$, while
by induction $\Phi'$ is open in $C\mid_{\Sigma}$, so $E\mid_V\to C_V$
is open, from which the (open) image of $V$ in $S$ 
achieves the next step in the induction.
\end{proof}

\bigskip

En passant this resolves any patching difficulties, since:

\bigskip

\begin{bonus}\thlabel{flatC:fact3}
Let everything be as in \thref{flatC:lem1}, and denote by a bar the compactification of op. cit. then if \eqref{flatC.9} respects a smooth, or indeed flat, base change $T \to S$, the natural map,
\begin{equation}
\label{flatC.5}
(\overline E) \times_S T \to \overline{E \times_S T}
\end{equation}
is an isomorphism.
\end{bonus}

\begin{proof}
By hypothesis we have the base change,
\begin{equation}
\label{fix251}
\xymatrix{
G \times_S T \ \ar[d]_{\gamma_T} \ar@{^{(}->}[r]_{\kappa_T} &P \times_S T \ar[d]^{\pi_T} \\
E \times_S T \ar[r]_{\varphi_T} &T
}
\end{equation}
of \eqref{flatC.9}, so that the right hand side of \eqref{flatC.5} is the base change of \eqref{flatC.10} whenever $T/S$ is flat.
\end{proof}




%% file: 4Compactifying_Flatificator.tex
\section{Compactifying the flatifier}\label{S:end}

It will emerge that the problem of compactifying the flatifier and that of functorial flattening by blowing up are one and the same, since the key operation is,

\begin{defn}\thlabel{end:defn1}
Let ${\mathscr F}$ on ${\mathcal X} \xrightarrow{f} S$ be as in \thref{flatG:fact2} with $S$ consistent in a Japanese way, \thref{con:defn1}, and suppose the flatifier $F_{\mathscr F} \to S$ exists (i.e. op. cit. the non-flatness of ${\mathscr F}$ is universally pure) 
and that the fitted flatifier, $E_{\mathscr{F}}\hookrightarrow F_{\mathscr{F}}$,
of \thref{flatP:fact4} also exists (e.g. 
the further condition of $Z_f(\mathscr{F})/S$ proper in the non-discrete case)
while admitting a, to be specified, compactification $E_{\mathscr{F}} \mapsto \overline E_{\mathscr{F}}$ respecting smooth base change, i.e. \eqref{flatC.5} holds, then the $0$th Fitting ideal, $J'$, of the aforesaid compactification, equally respects smooth base change and determines a functorial {\it blow up} (without support in a minimal prime),
\begin{equation}
\label{end.1}
S \overset{\rho_S}{\longleftarrow} S^{\#} = {\rm Bl}_J'(S)
\end{equation}
{\it in the fitted flatifier} with exceptional divisor $D$ and proper transform of ${\mathscr F}$ the co-kernel of the exact sequence,
\begin{equation}
\label{end.2}
0 \to {\mathscr H}_D^0 (S^{\#} , \rho_S^* \, {\mathscr F}) \to \rho_S^* \, {\mathscr F} \to {\mathscr F}^{\#} \to 0 \, .
\end{equation}
In particular, by \thref{flatC:fact3}, if the blow up in the fitted flatifier exists on some cover by affinoids of a  champ in the smooth topology, then it exists globally, so, for such more general $S$ we will say that the blow up in the fitted flatifier exists, if it exists everywhere locally.
\end{defn}

Now, from what we already know we may easily deduce,

\begin{fact}\thlabel{end:fact1}
Suppose, under the hypothesis of \thref{end:defn1} that the blow up in the fitted flatifier always exists then for ${\mathcal S}$ a formal quasi-compact locally Noetherian champ in the smooth topology, the sequence of blow ups and proper transforms,
\begin{equation}
\label{end.3}
{\mathcal S}_0 = {\mathcal S} \, , \ {\mathscr F}_0 = {\mathscr F} \, , \ {\mathcal S}_{i+1} = {\mathcal S}_i^{\#} \, , \ {\mathscr F}_{i+1} = {\mathscr F}_i^{\#} \, , \quad i \geq 0,
\end{equation}
wherein \thref{end:defn1} is supposed for every $\mathscr{F}_i/\mathcal{S}_i$,
terminates for $m \gg 0$ in a champ $\widetilde{\mathcal S} = {\mathcal S}_m$, and a sheaf $\widetilde{\mathscr F} = {\mathscr F}_m$ such that the flatifier $F_{\widetilde{\mathscr F}} \to \widetilde{\mathcal S}$ is an embedded closed sub-champ with the same reduction, or, equivalently, \cite{ega4}[Th\'eor\`eme 11.4.1] and \thref{con:defn1}, $\widetilde{\mathscr F}$ is flat over $\widetilde{\mathcal S}_{\rm red}$.
\end{fact}

\begin{proof}
By the functoriality of blowing up in the fitted flatifier and the quasi-compactness of ${\mathcal S}$, we may, without loss of generality, suppose that ${\mathcal S}$ is an affinoid $S = {\rm Spf} (A)$, or, indeed, a sufficiently small \'etale neighbourhood of the same. Furthermore if $C$ is  any compactification of $E_{\mathscr{F}}$ then, for $f(\vert Z\vert_f(\mathscr{F}))$ as in \thref{flatG:def2}, everywhere \'etale locally,
over a point, $s \in f(\vert Z\vert_f(\mathscr{F}))$, 
it decomposes as a disjoint union of open and closed pieces,
\begin{equation}
\label{end.4}
E_{\mathscr F}^* \coprod ({\rm stuff})
\end{equation}
where $E_{\mathscr{F}}^*$ is the pointed flatifier through a 
sufficiently small \'etale neighbourhood at $s \in S$ of 
the components $V'$, \thref{flatP:fact4}, where the pointed
flatifier isn't supported on a minimal prime. As such, \'etale locally around any $s$, the fitting ideal of the compactification decomposes as a product,
\begin{equation}
\label{end.5}
J_s^* \cdot ({\rm stuff})
\end{equation}
where $J_s^*$ is the ideal of the pointed flatifier in $V'$ through $s$. Thus, again, \'etale locally, the blow up in the fitted flatifier factors through the blow up in the pointed flatifier $J_s^*$. Better still,
\thref{flatP:fact5} applies to conclude that the surjection,
\begin{equation}
\label{end.6}
{\mathscr F} \mid_{{\mathcal X}_{s^{\#}}} \twoheadrightarrow {\mathscr F}^{\#} \mid_{{\mathcal X}_{s^{\#}}}
\end{equation}
is non-trivial at some point $x \in {\mathcal X}_{s^{\#}}$. Now suppose that \thref{end:fact1} is false, then by the compactness of the Zariski-Riemann surface of $S$ (or just Tychonoff for the product of the sequence \eqref{end.3}) there is a sequence of points $s_{i+1}$ in each exceptional divisor $D_{i+1}$, $i \geq 0$, on $S_{i+1}$ such that,
\begin{equation}
\label{end.7}
s_{i+1} \to s_i \, .
\end{equation}
Equally, since any map of fields is an injection, we can suppose that for all $i \gg 0$, $k(s_i) \to k(s_{i+1})$ is algebraic, so choosing an algebraic closure $K$ of some $k(s_m)$, $m \gg 0$, we can, for $i \geq m$, identify \eqref{end.7} with a sequence of $K$-points and the fibres ${\mathcal X}_{s_i}$ with the same fibre ${\mathcal X}_K$. Consequently by \eqref{end.6} we get a strict sequence of quotients of coherent sheaves,
\begin{equation}
\label{end.8}
{\mathscr F}_m \mid_{{\mathcal X}_K} \twoheadrightarrow {\mathscr F}_{m+1} \mid_{{\mathcal X}_K} \twoheadrightarrow {\mathscr F}_{m+2} \mid_{{\mathcal X}_K} \twoheadrightarrow \cdots
\end{equation}
which contradicts the local Noetherianity and quasi-compactness of ${\mathcal X}$.
\end{proof}

In order to apply this we will, plainly need,

\begin{fact}\thlabel{end:fact2}
Let things be as in \thref{end:defn1} 
and suppose moreover that $Z_f(\mathscr{F})/S$, \thref{flatG:defn100}, 
is proper over $S$, and the latter is
universally consistent in a Japanese way, \thref{con:defn1},
then the blow up in the fitted flatifier exists.
\end{fact}

More precisely we will prove by induction on 
$m\geq 0$,

\begin{fact}\thlabel{end:fact22}
Let everything be as in \thref{end:fact2}, with,
\begin{equation}\label{221}
\emptyset=E_0 \subset E_1 \subset \cdots E_n\subset E_{n+1}=E_{\mathcal{F}}
\end{equation}
the functorial filtration by closed subsets obtained by applying \thref{flatG:fact66}
to the structure sheaf, then for every $0\leq m\leq n$, each $E_i$, $0\leq i\leq m+1$ 
admits a compactification respecting smooth base change, \eqref{flatC.5}.
\end{fact}

\begin{proof}[proof of \thref{end:fact22}]
By induction on $\dim S$, so, to begin with, in dimension $0$, whence, 
$E_{\mathscr{F}}$ is empty, so this is ok. Otherwise, the 
properness condition on $Z_f(\mathscr{F})/S$ ensures the same 
whether on any modification of it, 
or any closed sub-champ thereof.
As such, with a view to doing the $m=0$ of \thref{end:fact22},
by downward induction on the dimension of $E_{\mathscr{F}}$,
\thref{end:fact1} ensures that there is a sequence of blow ups in centres which factor through 
$\mathfrak{Z}$, \thref{flatG:defn100}, and its proper transforms, are not supported in a minimal prime of the same, and result in a final modification $\widetilde S \xrightarrow{\pi} S$ together with closed sub-schemes,
\begin{equation}
\label{end.11}
\widetilde{\mathfrak{Z}} \hookrightarrow \widetilde S_{\rm red} \hookrightarrow \widetilde S
\end{equation}
where, a priori, only the proper transform of ${\mathscr F} \vert_{\mathfrak{Z}}$ 
is flat, but by a minor variation of \thref{flatP:fact5} \& \eqref{end.6} et seq.,
the proper transform $\widetilde{\mathscr F}$ on all of $\widetilde{S}$
restricted to $\widetilde{\mathfrak{Z}}$
is flat, and we assert, 
\begin{claim}\label{end:claim101}
The union, $P$, of the irreducible components of $(E_{\widetilde{\mathscr{F}}})_1$, \eqref{221},
containing $\widetilde{\mathfrak{Z}}$, \thref{rmk:jap1},
embed into  $\widetilde S$, has reduction $\widetilde{\mathfrak{Z}}$ and
for $G\hookrightarrow P$, the open where $\mathscr{F}\mid_P$ is flat, 
the natural map $G\rightarrow (E_{\mathscr{F}})_1$ dominates every component
whose reduction doesn't factor through the locus,
$\mathfrak{Z}_1\hookrightarrow\mathfrak{Z}$,
where $\widetilde{\mathfrak{Z}}$
modifies $\mathfrak{Z}$
\end{claim}
\begin{proof}[proof of \thref{end:claim101}]
Let $z\in\widetilde{\mathfrak{Z}}$ lie over $s\in\mathfrak{Z}$ with $V_{\mathrm{red}}=V'\cup V''$ the
components of \thref{flatP:fact4} where $\mathscr{F}_s$ isn't, resp. is, flat.
Thus, \thref{blow:fact3}, the \'etale neighbourhood $\widetilde{V}\to \widetilde{S}$
afforded by base change has reduction the union of the proper transforms,
$\widetilde{V'}$ and $\widetilde{V''}$. Better still: $\mathscr{F}$ is flat over
the latter, thus, inter alia, there is no torsion, so, cf. \eqref{end.2},
$\widetilde{\mathscr{F}}_{\widetilde{V''}}=\mathscr{F}_{\widetilde{V''}}$ is flat too,
and, in the obvious notation, $\widetilde{V}''\supseteq \widetilde{V''}$.
Similarly, at least in the discrete case, $\widetilde{\mathscr{F}}=\mathscr{F}$,
at every generic point of $\widetilde{\mathfrak{Z}}$, thus, at such points, 
the pointed flatifier of  $\widetilde{\mathscr{F}}_{\widetilde{V'}}$ factors through the
total transform of $\mathfrak{Z}$, which is nowhere dense, and whence
$\widetilde{V}'\supseteq \widetilde{V'}$, so, altogether,
$\widetilde{V}'= \widetilde{V'}$.
Now, a priori, this latter argument risks being meaningless in 
the non-discrete case, so, in general we argue as follows: 
supposing $V$ sufficiently fine, 
fix
an irreducible component $W$  of $V'$, 
then, by definition, the pointed flatifier of $\mathscr{F}_W$ isn't
schematically dense, so, cf. \eqref{flatP.8}, there is an $x\in \mathcal{X}$ such
that the flatifier, $F^x$, of $\mathscr{O}_{W,s}\to\mathscr{O}_{\mathcal{X},x}$
isn't schematically dense in $W$. 
As such if $\widetilde{W}$ is the proper transform of $W$,
then the pointed flatifier of $\mathscr{F}_{\widetilde{W}}$, factors
through the total transform of $F^x$, which, cf. \eqref{rank55}, cannot
be supported on a minimal prime, so, again $\widetilde{V}'= \widetilde{V'}$.
In particular, therefore, $\widetilde{\mathfrak{Z}}$ embeds in $E_{\widetilde{\mathscr{F}}}$
as a union of a subset of the irreducible components of $(E_{\widetilde{\mathscr{F}}})_{\mathrm{red}}$
by \thref{blow:fact3}, and we take $P$ to be resulting subspace of
$(E_{\widetilde{\mathscr{F}}})_1$ supported on $\widetilde{\mathfrak{Z}}$ defined by \thref{rmk:jap1}.
We know, however, that $(E_{\widetilde{\mathscr{F}}})_1\to\widetilde{S}$ is net, 
so $P$ embeds in $\widetilde{S}$ and defines a closed sub-scheme $\pi(P)\hookrightarrow S$
by \cite{ega3}[Th\'eor`eme 3.4.2]. Consequently, 
$\pi(P)_1$ is the Zariski closure in $S$ of $E_1$,
by \cite{ega4}[Th\'eor\`eme 11.4.1], and $E_1\to\pi(P)_1$
is almost a modification by op. cit. \& \thref{flatP:fact6}.
Finally, therefore, observe, that the proper transform of anything must, however, factor
through the proper transform of the closure of it's image, 
so we have an almost modification $\widetilde{E}_1\to P$,
which is an isomorphism over the open,
$G$, where $\mathscr{F}_P$ is flat, and, whence, $G$ dominates every component of $E_1$ which doesn't factor 
through the locus where 
$\widetilde{E}_1$ modifies $E_1$.
\end{proof}
As such, in the notation of \thref{end:claim101} we have a commutative square,
\begin{equation}
\label{end.12}
\xymatrix{
G \ \ar[d]_{\gamma} \ar@{^{(}->}[r]_{\kappa} &P \ar[d]^{\pi} \\
\varphi:E_1\backslash\{\text{components in $\mathfrak{Z}_1$}\}\ar[r]^{} &S
}
\end{equation}
in which $\gamma$ is dominant, $\kappa$ is open, and $\pi$ is proper, so we'll be able to conclude that the bottom left corner of  
\eqref{end.12}
is compactifiable  by \thref{flatC:lem1}, if $\gamma$ is proper. To this end observe,
\begin{claim}\thlabel{end:claim1}
Over the above open $G$, the natural surjective map,
\begin{equation}
\label{end.13}
{\mathscr F}\!\mid_G \, \twoheadrightarrow \widetilde{\mathscr F} \mid_G 
\end{equation}
is an isomorphism.
\end{claim}

\begin{proof}[Proof of \thref{end:claim1}]
We require to prove that the kernel of \eqref{end.13} is zero. It is, however, by a cursory inspection of the long exact sequence for ${\rm Tors}_{\bullet}^{{\mathscr O}_G} (M , \ )$, implied by \eqref{end.13}, for any ${\mathscr O}_G$ module $M$, flat over $G$ while being zero at every generic point, so it's zero.
\end{proof}

Consequently if we have a traite $T \xrightarrow{ \, \tau \, } P$ whose generic point is in $G$ and $\pi \tau$ factors through $E_{\mathscr{F}}$, then we have a surjection of flat $T$-sheaves,
\begin{equation}
\label{end.14}
{\mathscr F} \mid_T \ \twoheadrightarrow \widetilde{\mathscr F} \mid_T
\end{equation}
which is an isomorphism on the generic point, so $\tau$ factors through $G$ and $\gamma$ of \eqref{end.12} is proper. At this point the only issue in concluding to a a compactification
of all of $E_1$ is how to deal with any components whose reduction
might factor through $\mathfrak{Z}$. Certainly, we could just appeal
to \thref{flatP:fact6} and repeat \thref{end:claim101} with $\mathfrak{Z}$
replaced by the closure of their image, but the way to do it, which
respects smooth base change, is to flatten $\mathscr{F}_{\mathfrak{Z}_1}$,
by way of $\widetilde{\mathfrak{Z}}_1\to \mathfrak{Z}_1$, so that 
taking proper transforms picks up any components whose reduction
doesn't factor through the locus $\mathfrak{Z}_2\hookrightarrow
\mathfrak{Z}_1$ which supports the modification, then do \eqref{end.12}
et seq. again,  then flatten $\mathscr{F}_{\mathfrak{Z}_2}$ etc..

Now before attempting to conclude the induction in $m$ of \thref{end:fact22},
observe that the first stage of flattening $\mathscr{F}_{\mathfrak{Z}}$ achieves
not just the part of $E_1$ afforded by the left hand corner of \eqref{end.12},
but, by the same argument, also the embedded primes of the said components which 
don't factor through the aforesaid $\mathfrak{Z}_1$, or equivalently, since 
the formation of the latter respects smooth base change, we've actually functorially
compactified the larger sub-scheme, $E_1^+$ say, of \thref{rmk:jap1}, defined
by the associated primes not contained in $\mathfrak{Z}_1$. In the above notation,
we finish, therefore, by way of the following refinement of our inductive proposition,

\begin{claim}\label{end:fact23}
Let everything be as in \thref{end:fact22}, then for the above decreasing chain
of (everywhere properly included) closed sub-schemes,
$\mathfrak{Z}_{\bullet}$,
respecting smooth base change, with, 
\begin{equation}\label{222}
\emptyset=E_0^+ \subset E_1^+ \subset \cdots E_n^+\subset E_{n+1}^+=E_{\mathcal{F}}
\end{equation}
the functorial filtration it affords, \thref{rmk:jap1},
by the rule no associated prime of $E_i^+$ factors through $\mathfrak{Z}_i$;
then for every $0\leq m\leq n$, each $E^+_i$, $0\leq i\leq m+1$ 
admits a compactification respecting smooth base change, \eqref{flatC.5}.
\end{claim}
We've just done $m=0$, 
and, to go from $m$ to $m+1$, we initially follow the strategy of \thref{end:claim101}
in flattening $\mathscr{F}_{\mathfrak{Z}_{m}}$ by way of 
$S_m\hookleftarrow \widetilde{\mathfrak{Z}}_{m}\to
\mathfrak{Z}_{m}$, i.e. the restriction $\mathscr{F}_m\mid_{\widetilde{\mathfrak{Z}}_m}$
of the proper transform is flat. Consequently, \cite{ega4}[Th\'eor\`eme 11.4.1], applies as
in the proof of \thref{end:claim101} to conclude that
any associated primes of $E_{\mathscr{F}}$ which don't factor through $\mathfrak{Z}_m$
has a well defined Zariski closure which we'll employ in the form
that we have an open embedding,
\begin{equation}\label{223}
\widehat{E}_{\mathscr{F}}\backslash{\mathfrak{Z}_{m+1}}\hookrightarrow \widehat{E}_{\mathscr{F}_m},
\end{equation}
wherein $\widehat{\bullet}$ is completion in the trace of 
$\mathfrak{Z}_m$, resp. $\widetilde{\mathfrak{Z}}_m$.
Equally, if we apply the inductive proposition to construct a compactification
$P_m\to S_m$ 
of the fitted flatifier of $\mathscr{F}_m$, up to level $m$, then,
by construction the completion, $\widehat{P}_m$ of $P_m$ in the (obvious) section over
$\widetilde{\mathfrak{Z}}_m$ is,
\begin{equation}\label{224}
(\widehat{E}_{\mathscr{F}_m})_m^+
\end{equation}
which is a closed formal sub-scheme of the right hand side of \eqref{223}, 
so we can form (as sheaves over $S_m$ is fine) the push-out,
\begin{equation}\label{225}
(\widehat{E}_{\mathscr{F}_m})_m^+\rightrightarrows P_m\coprod \widehat{E}_{\mathscr{F}_m}
\rightarrow P
\end{equation}
which differs from $P_m$ in exactly the extra nilpotent structure afforded
by the difference between \eqref{224} and the right hand side of \eqref{223}.
In particular the open subset $G\hookrightarrow P$ where $\mathscr{F}$ is
flat now (left hand inclusion in \eqref{223}) dominates $E_{m+1}^+\hookrightarrow E_{\mathscr{F}}$ defined by
associated primes which don't factor through $\mathfrak{Z}_{m+1}$. As such
we can apply \thref{flatC:lem1} exactly as in \eqref{end.12} et seq. to
conclude.
\end{proof}
Putting this together we have, therefore, the obvious corollary,

\begin{cor}\thlabel{end:cor1}
Let $f : {\mathcal X} \to {\mathcal S}$ be a quasi-compact map of formal champ in the smooth topology,
with $\mathcal{S}$ universally consistent in a Japanese way, \thref{con:defn1},
and ${\mathscr F}$ a coherent sheaf on ${\mathcal X}$ enjoying $Z_f(\mathscr{F})/\mathcal{S}$ proper, then the sequence of blow ups in the fitted flatifier of \eqref{end.3} exists and leads to a modification $\rho : \widetilde{\mathcal S} \to {\mathcal S}$ by way of a finite sequence of blow ups (nowhere supported in a minimal prime) such that the flatifier $F_{\widetilde{\mathscr F}} \to \widetilde{\mathcal S}$ of the proper transform of ${\mathscr F}$ is an embedded closed sub-champ with the same reduction.
\end{cor}

\begin{proof}
Immediate by \thref{end:fact1} and \thref{end:fact2}.
\end{proof}

Of which a pleasing consequence is,

\begin{cor}\thlabel{end:cor2}
Let everything be as in \thref{end:cor1}, and suppose moreover that $f$ is
proper, then $f$ is surjective in the sense of \thref{A:mod}.(3) iff
$\mathscr{O}_{\mathcal{S}_{\mathrm{red}}}\to f_*\mathscr{O}_{\mathcal{X}_{\mathrm{red}}}$
is injective.
\end{cor}
\begin{proof}
The only if direction is immediate. Conversely, without loss of generlity,
everything is reduced by \thref{blow:fact3}, so, in
applying \thref{end:cor1}, we have a commutative diagram,
\begin{equation}\label{226}
\begin{CD}
{\mathcal{X}}@<\rho<<\widetilde{\mathcal{X}}\\
@V f VV @VV f V \\
{\mathcal{S}}@<\rho<<\widetilde{{\mathcal{S}}},
\end{CD}
\end{equation}
in which the horizontals are nowehere schematically dense blow ups, 
and the righmost vertical is flat. In particular,
whether by hypothesis or, since the $\rho$'s are algebraisable, by, 
\cite{knutson}, we have injections,
\begin{equation}\label{227}
\mathscr{O}_{\mathcal{S}}\hookrightarrow 
f_*\mathscr{O}_{\mathcal{X}}\hookrightarrow f_*\rho_*=\rho_*f_* \mathscr{O}_{\widetilde{\mathcal{X}}}
\end{equation}
As such, if the injectivity condition of \thref{end:cor2} fails
for the rightmost $f_*$ in \eqref{226}, then by the algebraisability
of the $\rho$'s and \cite{ega3}[Th\'eor\`eme 3.4.2], there is a
coherent sheaf on $\widetilde{\mathcal{S}}$ with proper support
which dominates $\mathcal{S}$. This is, however, plainly nonsense,
and since the lower, like the upper, $\rho$ is injective in the sense 
of \thref{A:mod}, we may, therefore suppose that $f$ is flat, and
$\mathcal{S}$ is the formal
spectrum of a D.V.R.. Now over the closed point, $s$, choose $x$ in the 
(non-empty open) locus where the fibre $\mathcal{X}_s$ is CM. As such,
the liftings, $t_i$, of a set of parameters from the fibre to $\mathscr{O}_{\mathcal{X},x}$
define a regular sequence, so that $t_i=0$ is flat and finite over $R$,
and is, therefore, the required lifting.
\end{proof}




%% file: 5Algebraisation_Formal_Deformations.tex
\section{Algebraisation of formal deformations}\label{S:alg}

Our goal is criteria for the algebraisation of formal deformations of algebraic spaces, and to which end we introduce our,

\begin{setup}\thlabel{alg:setup1}
Throughout this section $A$ will be a Noetherian ring complete in an $I$-adic topology which is universally consistent in a Japanese way, \thref{con:defn1}, with, 
\begin{equation}
\label{alg1}
A_m = A / I^{m+1}\, , \ S_m = {\rm Spec} \, A_m \hookrightarrow S = {\rm Spf} (A) \, , \quad m \geq 0 \, ,
\end{equation}
and $\pi : {\mathfrak X} \to S$ a proper formal algebraic space with special fibre $X_0$, while, more generally, for a coherent sheaf ${\mathscr F}$ on ${\mathfrak X}$ we put,
\begin{equation}
\label{alg1.bis}
{\mathscr F}_m = {\mathscr F} \times_S S_m \, .
\end{equation}
\end{setup}

Our goal is to establish algebraisation by counting sections of sheaves, and to this end we make,

\begin{defn}\thlabel{alg:defn1}
Let ${\mathfrak X} / S$ be as in \thref{alg:setup1} then a line sheaf on ${\mathfrak X}$ is a coherent sheaf ${\mathscr L}$ with generic rank $1$, in the sense of \thref{rank:fact1}.
\end{defn}

With this notation our initial goal is,

\begin{babyCrit}\thlabel{alg:baby1}
For each reduction of each irreducible component, there is a proper modification, \thref{A:mod}, $S_i \to S$ of the image such that for $\pi_i : {\mathfrak X}_i \to S_i$ the proper transform of the component and $d_i = \dim {\mathfrak X}_i - \dim S_i$ there is a line sheaf ${\mathscr L}_i$ on ${\mathfrak X}_i$ for which the generic rank, $r_i(n)$, of $\pi_* \, {\mathscr L}_i^{\otimes n}$ satisfies the growth estimate,
\begin{equation}
\label{alg1.new}
\liminf_{n \to \infty} \ n^{-d_i} r_i (n) > 0 \, .
\end{equation}
\end{babyCrit}

In this generality there is a slew of things that threaten not to have sense, so let's make;

\begin{CheckList}\thlabel{alg:list1}
The precise definition/demonstration of well definedness of the hypothesis of \thref{alg:baby1} are,

\smallskip

\noindent (A) For any closed subspace ${\mathfrak Y}$ of ${\mathfrak X}$, $\pi_* {\mathscr O}_{\mathfrak Y}$ is a finite ${\mathscr O}_S$-module by \cite{ega3}[Théo\-rème 3.4.2], whose support is, by definition, the image of ${\mathfrak Y}$.

\smallskip

\noindent (B) By \thref{A:mod}, a modification comes equipped with a closed subspace $T_i \hookrightarrow S_i$ which étale locally has no support in an irreducible component off which we have an isomorphism in the sense of \thref{A:artin}. However, by hypothesis, ${\mathfrak X}_i$ and $S_i$ are reduced and globally irreducible, so, by \thref{rank:lem2} checking ``everywhere non-denseness'' is easy, i.e. iff $T_i \ne S_i$.

\smallskip

\noindent (C) In particular the proper transform ${\mathfrak X}_i$ of the initial component ${\mathfrak Y}_i$ is by definition the fibre modulo $T_i$ torsion, and this is still a modification, by another application of \thref{rank:lem2}, of ${\mathfrak Y}_i$ because the fibre over $T_i$ cannot dominate ${\mathfrak Y}_i$.

\smallskip

\noindent (D) Finally, therefore, \thref{rank:bonus1} applies to conclude that $\pi_* \, {\mathscr L}_i^{\otimes n}$ has constant rank $r_i (n)$ on any cover of ${\mathfrak X}_i$ by affinoids, and, better still, this is even equal to the rank of its completion at a generic point of ${\mathfrak X}_i$.
\end{CheckList}

In this context, albeit independently of \thref{alg:defn1}, we have,

\begin{lem}\thlabel{alg:lem0}
Let ${\mathfrak X}/S$ be as in \thref{alg:setup1} then ${\mathfrak X}$ is algebraisable iff its reduction ${\mathfrak X}_{\rm red}$, \thref{con:fact1}, is algebraisable. 
\end{lem}

\begin{proof}
In the presence of the consistency conditions implied by \thref{alg:setup1}, ${\mathfrak X}_{\rm red}$ is defined, \thref{con:fact1}, by a coherent sheaf of ideals on ${\mathfrak X}$, so, the converse follows from \cite{knutson} irrespectively of whether reduction commutes with completion.

Otherwise ${\mathscr O}_{\mathfrak X}$ enjoys a finite filtration by ideals $F^p = (F^1)^p$ with successive quotients coherent ${\mathscr O}_{{\mathfrak X}_{\rm red}}$ modules, and $F^1$ the nil radical. As such let ${\mathfrak X}_p \hookrightarrow {\mathfrak X}$ be cut out by $F^p$, and proceed by induction on $p \geq 1$. In particular we have a short exact sequence,
\begin{equation}
\label{alg2.new}
0 \longleftarrow {\mathscr O}_{{\mathfrak X}_p} \longleftarrow {\mathscr O}_{{\mathfrak X}_{p+1}} \longleftarrow J \longleftarrow 0 \, , \quad J^2 = 0 \, ,
\end{equation}
which more likely than not fails to be a short exact of coherent ${\mathscr O}_{{\mathfrak X}_p}$-modules so we have some work to do to reduce to \cite{knutson}, to wit: 

\begin{claim}\thlabel{alg:claim0}
Let ${\mathfrak X}$ be a formal algebraic space, or a Deligne-Mumford champ, formally of finite type over a base $S$, $J$ a coherent sheaf of ${\mathscr O}_{\mathfrak X}$ modules, and $U \mapsto E(U)$ the pre-sheaf of sets whose elements are ring extensions of the form,
\begin{equation}
\label{alg21.new}
0 \longleftarrow {\mathscr O}_U \longleftarrow {\mathscr O}_Y \longleftarrow J \longleftarrow 0 \, , \quad J^2 = 0
\end{equation}
then,
\begin{enumerate}
\item[(1)] $E$ has the structure of an ${\mathscr O}_{\mathfrak X}$ module, and the associated sheaf, ${\mathscr E}$, is a coherent ${\mathscr O}_{\mathfrak X}$ module.
\item[(2)] The obstruction to an element of ${\rm H}^0 ({\mathfrak X} , {\mathscr E})$ defining a global section of $E$ is given by an ${\mathscr O}_S$-linear map,
\begin{equation}
\label{alg22.new}
{\rm obs} : {\rm H}^0 ({\mathfrak X},{\mathscr E}) \longrightarrow {\rm H}^2 ({\mathfrak X} , {\mathcal H}om_{{\mathscr O}_{\mathfrak X}} (\Omega_{{\mathfrak X}/S} , J)) \, .
\end{equation}
\item[(3)] The kernel of the obstruction group in \eqref{alg22.new} is formally (i.e. could be empty) a principal homogeneous space under,
\begin{equation}
\label{alg23.new}
{\rm H}^1 ({\mathfrak X} , {\mathcal H}om_{{\mathscr O}_{\mathfrak X}} (\Omega_{{\mathfrak X}/S} J)) \, .
\end{equation}
\end{enumerate}
\end{claim}

\begin{proof}[Proof of \thref{alg:claim0}]
The ${\mathscr O}_{\mathfrak X}$ module structure on $E$ is much the same as that for coherent sheaves, i.e. if ${\mathscr O}_{Y_1} , {\mathscr O}_{Y_2}$ are extensions then their sum is,
\begin{equation}
\label{alg24.new}
{\mathscr O}_{Y_1} \oplus {\mathscr O}_{Y_2} := {\mathscr O}_{Y_1} \times_{{\mathscr O}_\mathfrak{X}} {\mathscr O}_{Y_2} / \Delta^- (J) \, , \quad \Delta^- (x) = (x,-x)
\end{equation}
wherein $\times_{{\mathscr O}_{\mathfrak X}}$ is fibre product of rings, while the multiplication of an extension ${\mathscr O}_Y$ by $f$ is,
\begin{equation}
\label{alg25.new}
f \bullet {\mathscr O}_Y := {\mathscr O}_Y [J] / (j,-fj) \, .
\end{equation}
In order to check that the associated sheaf, ${\mathscr E}$, is coherent, one uses the finite type hypothesis to write ${\mathscr O}_{\mathfrak X}$ locally as a quotient of,
\begin{equation}
\label{alg26.new}
{\mathscr O}_{P} := {\mathscr O}_S \{T_1 , \cdots , T_d\} ({\rm modulo} \, I)
\end{equation}
so that any extension ${\mathscr O}_Y$ of the form \eqref{alg21.new} is the quotient of the trivial extension,
\begin{equation}
\label{alg27.new}
{\mathscr O}_P [J] / (x,\varepsilon (x) \mid x \in I) \, , \quad \varepsilon \in {\mathcal H}om_{{\mathscr O}_{\mathfrak X}} (I,J)
\end{equation}
while maps from ${\mathscr O}_{P}$ to ${\mathscr O}_{P}[J]$ are,
\begin{equation}
\label{alg28.new}
{\mathcal H}om_{{\mathscr O}_{\mathfrak X}} (\Omega_{P/S} \mid_{\mathfrak X} , J) \, .
\end{equation}

As such, locally, the ${\mathscr O}_{\mathfrak X}$ module structure on $E$ is exactly that of,
\begin{equation}
\label{alg29.new}
{\mathcal H}om_{{\mathscr O}_{\mathfrak X}} (I/I^2 , J) \mod {\mathcal H}om_{{\mathscr O}_{\mathfrak X}} (\Omega_{P/S} \mid_{\mathfrak X} , J) \, .
\end{equation}

As to item (2), by definition, a global section of ${\mathscr E}$ is an algebra of the form \eqref{alg21.new} over an open cover $U \to {\mathfrak X}$ together with an isomorphism,
\begin{equation}
\label{alg30.new}
g : p_1^* {\mathscr O}_Y \xrightarrow{ \ \sim \ } p_2^* {\mathscr O}_Y \, , \quad p_i : U \times_{\mathfrak X} U \to U
\end{equation}
of extensions of ${\mathscr O}_{U \times_{\mathfrak X} U}$. This will, however, only yield a global extension if we have the co-cycle condition,
\begin{equation}
\label{alg31.new}
p_{13}^* (g) = p_{12}^* (g) \, p_{23}^* (g) \, , \quad p_{ij} : U \times_{\mathfrak X} U \times_{\mathfrak X} U \to U \times_{\mathfrak X} U
\end{equation}
and the failure of \eqref{alg31.new} in, say, $p_1^* {\mathscr O}_Y$ is up to a choice in $\{1,2,3\}$ the obstruction \eqref{alg22.new}, while up to isomorphism, the possible solutions of \eqref{alg31.new} are a principal homogeneous space under \eqref{alg23.new}.
\end{proof}

Returning to the proof of \eqref{alg:lem0}, we know by \cite{knutson} that the category of coherent ${\mathscr O}_{{\mathfrak X}_p}$ modules is equivalent to that on the algebraisation, so ${\mathfrak X}_{p+1}$ is algebraisable by \eqref{alg2.new} and \thref{alg:claim0}.
\end{proof}

This will allow us to apply the main flattening theorem in establishing,

\begin{lem}\thlabel{alg:lem1}
Suppose there is a proper birational modification $f : {\mathfrak X}' \to {\mathfrak X}$, \thref{A:mod}, such that ${\mathfrak X}'$ is algebraisable, then ${\mathfrak X}$ is algebraisable.
\end{lem}

\begin{proof}
By \thref{alg:lem0} we may suppose that ${\mathfrak X}$ is reduced, and we proceed by induction on $\dim {\mathfrak X}$ with the case of dimension $0$ relative to $S$ being fairly trivial, i.e. anything finite over $S$ is algebraisable. Otherwise, by hypothesis, there is a closed formal subspace ${\mathfrak Z} \hookrightarrow {\mathfrak X}$ which does not contain any irreducible component and off which $\rho$ is \'etale and ``injective'' in the sense of Artin's criteria \thref{A:artin}. Now let ${\mathfrak Y} \hookrightarrow {\mathfrak X}'$ be the fibre over ${\mathfrak Z}$ then by \cite{knutson} not only is ${\mathfrak Y}$ algebraisable but so is,
\begin{equation}
\label{alg9.new}
{\mathfrak Y} \leftleftarrows {\mathfrak Y} \times_{\mathfrak Z} {\mathfrak Y} \hookrightarrow {\mathfrak X}' \times_{\mathfrak X} {\mathfrak X}' \hookrightarrow {\mathfrak X}' \times_A {\mathfrak X}' \, .
\end{equation}
As such there is a relation,
\begin{equation}
\label{alg10.new}
R \rightrightarrows Y
\end{equation}
in algebraic spaces over $A$ whose completion is the relation on the left of \eqref{alg9.new}. Plausibly, this doesn't lead to an algebraic space $R/Y$. We can however apply \thref{end:cor1}, and the fact that ${\mathcal Z}$ is reduced to get a diagram,
\begin{equation}
\label{alg11.new}
\xymatrix{
&{\mathfrak X}'_1 \ar[ld] \ar[d] & \ {\mathfrak Y}_1 \ar@{_{(}->}[l] \ar[ld] \ar[d] \\
{\mathfrak X}' \ar[d] &\Atop{\mbox{${\mathfrak Y}$}}{\mbox{${\mathfrak X}_1$}}\ar@{_{(}->}[l] \ar[ld] \ar[d] &{\mathfrak Z}_1 \ar@{_{(}->}[l] \ar[ld] \\
{\mathfrak X} &{\mathfrak Z} \ar@{_{(}->}[l]
}
\end{equation}
in which ${\mathfrak X}_1$ is ${\mathfrak X}$ blown up in the sequence of centres that lead to the flattening in the right most vertical face of \eqref{alg11.new}. In particular the relation ${\mathfrak Y}_1 \times_{{\mathfrak Z}_1} {\mathfrak Y}_1 \rightrightarrows {\mathfrak Y}_1$ is now not only algebraisable but flat so ${\mathfrak Z}_1$ is algebraisable by \cite{morikeel}, and whence ${\mathfrak Z}$ by our inductive hypothesis.

Consequently another application of \cite{knutson} tells us that the map ${\mathfrak Y} \to {\mathfrak Z}$ itself is algebraisable. To conclude, however, to the algebraicity of ${\mathfrak X}$ by \cite{artin} we need to know this for the fibres,
\begin{equation}
\label{alg12.new}
{\mathfrak Y}_m = \rho^{-1} ({\mathfrak Z}_m) \longrightarrow {\mathfrak Z}_m
\end{equation}
over every $m^{\rm th}$ thickening of ${\mathfrak Z} \hookrightarrow {\mathfrak X}$. Plainly we aim to simply repeat the above steps for ${\mathfrak Z}_0 = {\mathfrak Z}$, and the only point for care is that \eqref{alg12.new} can, \thref{end:cor1}, be flattened by blowing up in a sequence of nowhere dense centres iff it's flat over every generic point, which, indeed is the case since the limit of \eqref{alg12.new} is the fibre over the completion, $\widehat{\mathfrak Z}$, of ${\mathfrak X}$ in ${\mathfrak Z}$ which is reduced since all the local rings resulting from \thref{alg:setup1} are excellent.
\end{proof}

Which in turn we can employ to deduce,

\begin{lem}\thlabel{alg:lem2}
Let everything be as in \thref{alg:baby1}, then without loss of generality the line sheaf ${\mathscr L}$ is actually a line bundle satisfying \eqref{alg1.new}.
\end{lem}

\begin{proof}
We may apply \thref{end:cor1} to the line sheaf ${\mathscr L}$ of \thref{alg:defn1} and the identity map $f = {\rm id}_{\mathfrak X}$ to get a birational modification ${\mathfrak X}' \to {\mathfrak X}$ such that the proper transform ${\mathscr L}'$ of ${\mathscr L}$ is flat over ${\mathfrak X}'$ and whence, by item (B) of op. cit., it is a line bundle. Better still we have a natural injective map,
\begin{equation}
\label{alg6}
{\rm H}^0 ({\mathfrak X} , {\mathscr L}) \hookrightarrow {\rm H}^0 ({\mathfrak X}' , {\mathscr L}')
\end{equation}
so ${\mathscr L}'$ satisfies the estimate \eqref{alg1.new} because ${\mathfrak X}'$ and ${\mathfrak X}$ have the same irreducible components. As such, to conclude, we just resolve the base points of $\vert {\mathscr L}'\vert$ in the usual way.
\end{proof}

At which juncture we can give:

\begin{proof}[Proof of the \thref{alg:baby1}] 
The map,
\begin{equation}
\label{alg13.new}
\coprod_i {\mathfrak X}_i \longrightarrow {\mathfrak X}
\end{equation}
from the disjoint union of the irreducible components is a birational modification so by \thref{alg:lem0} and \thref{alg:lem1} we may suppose that ${\mathfrak X}$ is reduced and globally irreducible.

Similarly, we may, therefore also suppose that $S$ is a modification of ${\rm Spf} (A)$ of \thref{alg:setup1}. As such \thref{end:cor1} applies to find a flattening of $S \to {\rm Spf} (A)$ by a sequence of blow ups $\widetilde S \to {\rm Spf} (A)$ without support in a minimal prime, so the proper transform of $S$ along $\widetilde S \to {\rm Spf} (A)$ is $\widetilde S$ and $S$ is algebraisable by another application of \thref{alg:lem1}. Better still we might as well say $S = \widetilde S$, so, without loss of generality, it's even the $I$-adic completion of a scheme, while by a further application of \thref{end:cor1} we may, equally, suppose $\pi : {\mathfrak X} \to S$ is flat of relative dimension $d$.

In any case for each $n \geq 0$, the image of the natural maps,
\begin{equation}
\label{f251}
\pi_* \, L^n \otimes_{{\mathscr O}_S} {\mathscr O} \longrightarrow L^n
\end{equation}
define an ideal ${\mathscr I}_n$, so that for $\widetilde{\mathfrak X}_n \to {\mathfrak X}$ the blow up in the same we get maps,
\begin{equation}
\label{f252}
\xymatrix{
&\widetilde{\mathfrak X}_n \ar[ld] \ar[d] \\
{\mathfrak X} &P_n := {\mathbb P} (\widehat{\pi_* \, L^n})
}
\end{equation}
wherein by \cite{ega3}[Theorème 5.1.4] we can identify the universal widget for rank $1$ quotients of $\pi_* \, L^n$ in the formal category with the completion of the scheme, ${\mathbb P} (\pi_* \, L^n)$. Irrespectively the vertical arrow in \eqref{f252} is a proper map whose image has a well defined dimension $d_n \leq d$ and we assert,

\begin{claim}\thlabel{claim:f251}
The maximum $\delta$ over $n > 0$ of $d_n$ is $d$. 
\end{claim}

If we can prove \thref{claim:f251} then, more or less plainly, we're done. Indeed the base ${\mathscr I}_n$ of \eqref{f251} cannot vanish anywhere locally along a minimal prime by \thref{rank:lem2}, so the blow up in \eqref{f252} is a modification in the sense of \thref{A:mod}, thus we may apply \thref{alg:lem1} yet again to suppose that we have,
\begin{equation}
\label{alg14.new}
\xymatrix{
{\mathfrak X} \ar[d]_{\sigma} \\
{\mathcal Y} \ \ar@{^{(}->}[r] &P_n
}
\end{equation}
wherein, by \thref{claim:f251}, $\sigma$ is a map of relative dimension $0$ between reduced and irreducible varieties, so \thref{end:cor1} applies to find a flattening,
\begin{equation}
\label{alg15.new}
\xymatrix{
{\mathfrak X}' \ar[d]_{\sigma'} \ar[r] &{\mathfrak X} \ar[d]^{\sigma} \\
{\mathcal Y}' \ar[r] &{\mathcal Y}
}
\end{equation}
by a sequence of blowing ups. As such ${\mathcal Y}'$ is algebraisable, and $\sigma'$ is finite, so ${\mathfrak X}'$ is algebraisable, and whence ${\mathfrak X}$ by a final application of \thref{alg:lem1}.
\end{proof}

It therefore remains to prove \thref{claim:f251}, and despite appearances this isn't quite as obvious as it seems. The problem is that the scheme case of op. cit. appears to need closed points off the base locus, and, a priori we may not have any. As such we preceed by a series of simpler assertions beginning with,

\begin{subclaim}\thlabel{sub:f251}
Without loss of generality $d_n$ of \thref{claim:f251} is constant $<d$.
\end{subclaim}

\begin{proof}
Suppose \thref{claim:f251} is false, and replace $L$ by $L^n$ for some $n \gg 0$.
\end{proof}

\begin{subclaim}\thlabel{sub:f252}
If either we complicate the statement of \thref{alg:baby1} to allow not just $\pi_* \, L^n$ but any graded sub-algebra,
\begin{equation}
\label{261}
\coprod_n {\mathscr E}_n \longhookrightarrow \coprod_n \pi_* \, L^n
\end{equation}
with graded pieces of rank satisfying the growth estimate \eqref{alg1.new}, or, more conveniently, restrict ourselves to proving \thref{alg:baby1} rather than \thref{claim:f251}, we may suppose that the trace, $S_0 \hookrightarrow S$, is the support of a Cartier divisor which is everywhere of pure co-dimension 1.
\end{subclaim}

\begin{proof}
If the trace has support at a minimal prime somewhere then by \thref{rank:lem2} the topology of \thref{alg:setup1} is discrete, and we're done, while otherwise we may legitimately blow up in it by \thref{alg:lem1}. Plainly the growth estimate of \eqref{alg1.new} still holds after blowing up, albeit the original graded algebra may become a sub-algebra as in \eqref{261}.
\end{proof}

\begin{subclaim}\thlabel{sub:f253}
Without loss of generality none of the base loci, ${\mathscr I}_n$, of \eqref{f251} are supported on the trace $X_0$ of ${\mathfrak X}$.
\end{subclaim}

\begin{proof}
By \thref{sub:f252} and the flatness of ${\mathfrak X}/S$ we may identify $X_0$ not just with the Cartier divisor $\pi^* S_0$ but for every generic point $\xi$ of $X_0$, $X_0$ is reduced of pure co-dimension $1$. In particular, therefore, any minimal prime $D$ of ${\mathscr O}_{X,\xi}$ has a well defined order function whose value $v_D (X_0)$, resp. $v_D ({\mathscr I}_1)$, i.e. the minimum over a set of generators, satisfies
\begin{equation}
\label{262}
\infty > v_D (X_0) > 0 \, , \quad \mbox{resp.} \ \infty > v_D ({\mathscr I}_1) \geq 0
\end{equation}
since, otherwise, \thref{rank:lem2} would apply. Now by definition, ${\mathscr I}_1$ cannot be contained in ${\mathscr O}_{\mathfrak X} (-X_0)$, but it could have the same support, in which case there would be a minimal $e_1$ such that,
\begin{equation}
\label{263}
{\mathscr I}_1^{e_1} \subseteq {\mathscr O}_{\mathfrak X} (-X_0)
\end{equation}
and a corresponding maximal $n_1 \geq 1$ such that,
\begin{equation}
\label{264}
{\mathscr I}_1^{e_1} \subseteq {\mathscr O}_{\mathfrak X} (-n_1 \, X_0), \quad \mbox{but} \ {\mathscr I}_1^{e_1} \underset{\neq}{\subseteq} {\mathscr O}_{\mathfrak X} (-(n_1+1) \, X_0) \, .
\end{equation}
As such we'd have a submodule $\Gamma_2$ of $\pi_* \, L^{e_1}$ and an isomorphism,
\begin{equation}
\label{265}
\Gamma_1 := \pi_* \, L \longrightarrow \Gamma_2 (-n_1 \, X_0) : \ell \longmapsto \ell^{e_1} \, .
\end{equation}
In particular, therefore, we get a base locus defined by,
\begin{equation}
\label{266}
\pi^* \, \Gamma_2 \otimes {\mathscr O}_{\mathfrak X} \ -\!\!\!\twoheadrightarrow L^{e_1} \cdot \gamma_2
\end{equation}
which is either without support on $X_0$, or we get a new pair $e_2 , n_2$ satisfying,
\begin{equation}
\label{267}
\gamma_2^{e_2} \subseteq {\mathscr O}_{\mathfrak X}  (-n_2 \, X_0) \, , \ \mbox{etc.}
\end{equation}
As such were this situation to continue indefinitely, then we'd get a sequence of ideals $\gamma_q$, satisfying,
\begin{equation}
\label{268}
v_D (\gamma_q) = \frac{n_q}{e_q} \, v_D (X_0) + v_D (\gamma_{q+1})
\end{equation}
for any $v_D$ as found in \eqref{262}. By local Noetherianity and quasi-compactness, however, the $e_q$ are universally bounded, so, an infinite sequence satisfying \eqref{268} contradicts \eqref{262}, and whence on replacing $L$ by $L^{e_1 \cdots e_N}$, for some $N \gg 0$, we have that the initial base locus isn't supported on $X_0$.
\end{proof} 

\begin{subclaim}\thlabel{sub:f254}
Modulo the same caveat implied by \eqref{261} in the preliminaries to \thref{sub:f252} we can suppose that the vertical arrow in \eqref{alg14.new} (with $n=1$, by the way) is flat.
\end{subclaim}

\begin{proof}
Exactly as in \eqref{alg15.new} and with the notation therein, flatten the vertical in \eqref{alg14.new} along a sequence of blow ups $\rho : {\mathfrak Y}' \to {\mathfrak Y}$. In particular, therefore, there is a very ample divisor, $A$, on ${\mathfrak Y}'$ of the form,
\begin{equation}
\label{269}
{\mathscr O}_{P_1} (a) \mid_{{\mathfrak Y}'} (-E)
\end{equation}
wherein $a \gg 0$, and ${\mathscr O}_{{\mathfrak Y}'} (-E)$ is $\rho$-very ample. As such the pull-back of $A$ to ${\mathfrak X}'$ is,
\begin{equation}
\label{270}
{\mathscr O}_{{\mathfrak X}'} (aL - aB -E) 
\end{equation}
for $B$, the supposed, without loss of generality, Cartier base of the linear system $\pi_* \, L$. In particular, therefore, for $a \gg 0$, the sections of powers of,
\begin{equation}
\label{271}
L' := {\mathscr O}_{{\mathfrak X}'} (aL -E) 
\end{equation}
satisfy the growth estimate \eqref{alg1.new}. Plausibly it might well be true that the inclusion,
\begin{equation}
\label{272}
\pi_* \, A \longhookrightarrow \pi_* \, L'
\end{equation}
implied by \eqref{270} is strict, so $\sigma' : {\mathfrak X}' \to {\mathfrak Y}'$ is only a projection of the map to ${\mathbb P} (\pi_* \, L')$, but this is unimportant, since all we need is that ${\mathfrak Y}'$ achieves the maximal dimension in \thref{claim:f251}.
\end{proof}

\begin{subclaim}\thlabel{sub:f255}
Again with the caveat that the image ${\mathfrak Y}$ in \eqref{alg14.new}, may only be a sub-linear system of $\pi_* \, L$, albeit with the same dimension as the image of the latter, \eqref{272}, we may suppose that ${\mathfrak Y}/S$ is flat.
\end{subclaim}

\begin{proof}
Just apply the flattening theorem, \thref{end:cor1}.
\end{proof}

\begin{subclaim}\thlabel{sub:f256}
For any $n > 0$, let ${\mathfrak Y}_n$ be the closed image of $\widetilde{\mathfrak X}_n$ in \eqref{f252}, and ${\mathfrak Y}'_n \to {\mathfrak Y}_1$ the projection of the graph afforded by the Segre map,
\begin{equation}
\label{273}
{\rm Sym}^n (\pi_* \, L) \longrightarrow \pi_* \, L^n
\end{equation}
then the (possibly empty) locus where the fibres are positive dimensional never contains a component of any complete local ring of ${\mathfrak Y}_n$. Idem for the projection ${\mathfrak Y}_1 \to {\mathfrak Y}$ of \eqref{272}.
\end{subclaim}

\begin{proof}
By \cite{ega3}[Théorème 5.1.4] all of ${\mathfrak Y} , {\mathfrak Y}_n$ are algebraisable, so the locus where the fibres are positive dimensional is a well defined global subspace, and, whence, \thref{rank:lem2} applies to profit locally from global irreducibility.
\end{proof}

\begin{subclaim}\thlabel{sub:f257}
Quite plausibly with the caveat implied by \eqref{261} in the preliminaries to \thref{sub:f252} we may suppose that $S = {\rm Spf} (A)$ where $A$ is a complete local reduced 1-dimensional integral domain in it's ${\mathfrak m}$-adic topology.
\end{subclaim}

\begin{proof}
Our goal is the dimension calculation of \thref{claim:f251}, which we may safely perform at a generic point $\sigma$ of $S$. As ever, \thref{hensel:con1}, ${\mathscr O}_{S,\sigma}$ needn't be complete, but it's $I$-adic completion, $\widehat{\mathscr O}_{S,\sigma}$, is its ${\mathfrak m}(\sigma)$-adic completion, and the fact that this is reduced 1-dimensional follows from \thref{sub:f252} and \thref{con:defn1}. Now if replacing $S$ by ${\rm Spf} \, \widehat{\mathscr O}_{S,\sigma}$ were the only operation that we were performing then the linear systems wouldn't change by the theorem of formal functions. However, to achieve \thref{sub:f256} we may very well need to take fibres over,
\begin{equation}
\label{273bis}
{\rm Spf} \, \widehat{\mathscr O}_{S,\sigma} \longleftarrow \coprod_{\alpha} \, {\rm Spf} \, A_{\alpha} \, ,
\end{equation}
where $A_{\alpha}$ are the components which could well change the linear systems involved, so we should take ${\mathscr E}_n$ in \eqref{261} to be the restriction to (any) $A_{\alpha}$ of the global linear systems, $\pi_* \, L^n$, defined over $S$. Plainly this has no effect on the generic rank estimate of \eqref{alg1.new} by \thref{rank:bonus1}.
\end{proof}

With these rather extensive preliminaries we can now give,

\begin{proof}[Proof of \thref{claim:f251}]
Let $s \in S$ be the closed point in \thref{sub:f256}, $x$ a closed point of the trace missing the base point of the (sub)-linear system of \eqref{272} as guaranteed by \thref{sub:f253}, and $y$ its image in ${\mathfrak Y}$ of \thref{sub:f255}. Now observe that any projective variety over a field is generically CM (just taken a projection to a projective space), thus, so is any algebraic space by Chow's lemma. In particular, therefore, we can suppose that the trace $X_0 \hookrightarrow {\mathfrak X}$ is CM at ${\mathfrak X}$, and whence by \cite{matsumura}[Theorem 23.3] (albeit see the proof rather than the statement] and the flatness of ${\mathfrak X}/{\mathfrak Y}$, a system of parameters at $x$, defines ${\mathfrak Z} \hookrightarrow {\mathfrak X}$ flat over ${\mathfrak Y}$ which is proper over the $I$-adic completion $\widehat{\mathscr O}_{{\mathfrak Y},y}$, so, the composition,
\begin{equation}
\label{274}
\widehat{\mathscr O}_{{\mathfrak Y},y} \longrightarrow \widehat{\mathscr O}_{{\mathfrak X},x} \longrightarrow \widehat{\mathscr O}_{{\mathfrak Z},x}
\end{equation}
is finite into the bargain. Now suppose for some $n$,
\begin{equation}
\label{275}
\widehat{\mathscr O}_{{\mathfrak Y}_n,y_n} \longrightarrow \widehat{\mathscr O}_{{\mathfrak Z},x}
\end{equation}
isn't injective. The ring on the left of \eqref{275} is reduced, so this holds iff it holds for the reduction of the right hand side. Better still if ${\mathfrak p}$ is a minimal prime of the left hand side of \eqref{275}, then its intersection ${\mathfrak q}$ with $\widehat{\mathscr O}_{{\mathfrak Y},y}$ is also a minimal prime by \eqref{276}, over which there are finitely many minimal primes ${\mathfrak q}_i$ of ${\mathfrak Z}$. As such there are maps,
\begin{equation}
\label{276}
\xymatrix{
{\mathfrak q} \backslash \widehat{\mathscr O}_{{\mathfrak Y},y} \ar@{^{(}->}[d] \ar[r] &{\mathfrak p} \backslash \widehat{\mathscr O}_{{\mathfrak Y}_n,y_n} \ar[d] \\
\widehat{\mathscr O}_{\widehat{\mathfrak Z} , x / {\mathfrak q}_i} &\prod_i \widehat{\mathscr O}_{{\mathfrak Z},x} / {\mathfrak q}_i \ar@{->>}[l]
}
\end{equation}
so if even the composition of the rightmost vertical with the bottom horizontal wasn't injective, then the image of the left vertical would be a ring of smaller dimension rather than a finite subring. Finally therefore \eqref{275} can be understood as a sub-module of a finite $\widehat{\mathscr O}_{{\mathfrak Y},y}$ module, which for $n \gg 0$ must become constant. This means that the cost of a section of $\mathrm{H}^0 (X_0 , L^n)$ vanishing to order $m$ at $x$ is,
\begin{equation}
\label{277}
O(m^{\delta})
\end{equation}
independently of $n$. On the other hand $A$ in \thref{sub:f256} is a complete local ring, so it admits a resolution of singularities $\widetilde S \to S$ by a complete D.V.R., and we can replace ${\mathcal E}_n \otimes_{{\mathscr O}_S} {\mathscr O}_{\widetilde S}$ by its saturation in $\pi_* \, L^n$ to conclude that,
\begin{equation}
\label{278}
\mathrm{h}^0 (X_0 , L^n) \gg n^d
\end{equation}
by the growth condition \eqref{alg1.new}, from which $d=\delta$.
\end{proof}

Plainly, however, we require a verifiable alternative to the \thref{alg:baby1}, to wit

\begin{PractCrit}\thlabel{alg:real1}
Let ${\mathfrak X}/S$ be as in \thref{alg:setup1} and suppose that each irreducible component ${\mathfrak X}_i/S$ admits an open set $U_i$ which is $S$-flat of relative dimension $d_i$ while admitting a line sheaf ${\mathscr L}_i$, \thref{alg:defn1}, which further satisfies,

\medskip

(a) On some, and whence any, modification, \thref{A:mod}, $\varphi_i : \widetilde{\mathfrak X}_i \to {\mathfrak X}_i$ with exceptional locus $E$, $\varphi_i^* \, {\mathscr L}_i$ modulo $E$-torsion, cf. \eqref{end.2}, is a nef. line bundle $L_i$.

\medskip

(b) For $X_i \hookrightarrow {\mathfrak X}_i$ the trace, the map afforded by global sections,
\begin{equation}
\label{alg3}
\mathrm{H}^0 (X_i , {\mathscr L}_i \mid_{X_i}) \otimes_{{\mathscr O}_{S_0}} {\mathscr O}_{X_i} \longrightarrow {\mathscr L}_i
\end{equation}
embeds a generic point of the trace of $U_i$.
\end{PractCrit}

We may apply \thref{alg:lem1}, \thref{alg:lem2}, \eqref{alg3}, \thref{blow:fact3}, and, of course \thref{end:cor1} to obtain the following,

\begin{simplification}\thlabel{alg:simp1}
Without loss of generality,
\begin{enumerate}
\item[(A)] ${\mathfrak X} = {\mathfrak X}_i$, and whence drop the index $i$ from the notation.
\item[(B)] ${\mathscr L}$ is a line bundle, written $L$ for clarity.
\item[(C)] ${\mathfrak X}/S$ factors through a reduced algebraisable modification $S'$ and ${\mathfrak X}/S'$ is flat.
\item[(D)] ${\mathfrak X}$ is reduced.
\end{enumerate}
\end{simplification}

Plainly our goal is to verify \thref{alg:baby1} for $L$ of \thref{alg:simp1}, and in a minor variant of \thref{sub:f256} we can suppose that $S = {\rm Spf} (A)$ is a complete D.V.R. by \thref{rank:bonus1}. In particular at the closed point $s \in S$, we have control on the cohomology by,

\begin{revision}\thlabel{rev:f251}
The higher cohomology of a nef line bundle $L$ on the proper algebraic space $X_s$ of dimension $d$ admits the estimate,
\begin{equation}
\label{f281}
\mathrm{h}^q (X_s , L^n) \ll n^{d-1} \, , \quad \forall \, q > 0 \, .
\end{equation}
\end{revision}

Consequently we get the same estimate for the generic rank of $R^q \, \pi_* \, L^n$ by \thref{count:fact2}, while by flatness the Euler characteristic is constant in the sense of op. cit., so the generic rank of $\pi_* \, L^n$ admits the estimate,
$$
\mathrm{h}^0 (X_s , L^n) + O(n^{d-1})
$$
which by the flatness hypothesis in item (b) of \thref{alg:real1} remains of order $n^d$ even after the flattening operation of \thref{alg:simp1}.




%% file: 6Appendix.tex
\appendix
\section{Appendix}\label{S:A}

\subsection{Generalities on formal schemes}\label{SS:formal}

A priori formal schemes are built from admissible rings, \cite{ega0}[7.1.2] or \cite{formal}[1.2]. For our immediate, and highly Noetherian considerations, however, such a level of generality is not appropriate, since:

\begin{fact}\thlabel{formal:fact1}
Let $A$ be a Noetherian admissible ring, with $I$ its largest ideal of definition, \cite{ega0}[7.1.7], then the following are equivalent,
\begin{enumerate}
\item[(1)] $A$ is, topologically, $I$-adic.
\item[(2)] The graded ring $\underset{n \geq 0}{\coprod} I^n / I^{n+1}$ in its induced topology is discrete.
\item[(3)] Every ideal, $J$, is topologically finite generated, i.e. it is a topological quotient of $A^{\oplus m}$ for some $m \in {\mathbb Z}_{\geq 0}$.
\end{enumerate}
\end{fact}

\begin{proof}
(3) $\Rightarrow$ (1) By definition for any $n$, $I^n$ is a topological quotient of some $A^{\oplus m}$, thus $I^n / I^{n+1}$ is a topological quotient of $(A/I)^{\oplus m}$, but $I$ is open, so the latter topology is discrete.

\smallskip

(2) $\Rightarrow$ (1) By definition of induced discrete topology; for every $n$, $\exists$ an open $V_n$ such that,
\begin{equation}
\label{formal1}
I^n \cap V_n \subseteq I^{n+1} 
\end{equation}
so by induction on $n$, $I^n$ is open for any $n$.

\smallskip

(1) $\Rightarrow$ (3)  This is another rephrasing of the Artin-Rees lemma.
\end{proof}

In particular it follows that there is a mistake in EGA, to wit:

\begin{warning}\thlabel{formal:warn1}
The assertion of \cite{ega1}[10.6.5] that every Noetherian affine formal is adic is wrong. The cross reference to the non-trivial point in the proof to op. cit. 10.5.1 leads to op. cit. 10.3.6, which is circular since the latter supposes adic. A specific counterexample over a field $k$ is,
\begin{equation}
\label{formal2}
A := \frac{k[[y]][x]}{(x^2)}
\end{equation}
with a basis of neighbourhoods of zero given by the ideals,
\begin{equation}
\label{formal3}
J_n := (xy^n) \, , \quad n \geq 0 \, .
\end{equation}
As such: $f$ is topologically nilpotent iff it is nilpotent iff $x \mid f$, and, whence, $(x)$ is the ideal of definition, but the topology on $(x) = (x) / (x^2)$ is not discrete.
\end{warning}

This oversight notwithstanding our main flattening theorem, \thref{end:cor1}, employs Noetherianity endemically, and, while it might not be deadly, the failure of (3) in non-adic topologies is asking for problems that are best avoided, and whence we will adopt,

\begin{convention}\thlabel{formal:con1}
Unless stated otherwise any formal scheme, algebraic space, or champs which may occur in this article is to be supposed adic, i.e. it admits an atlas of the form,
\begin{equation}
\label{formal4}
\coprod_{\alpha} {\rm Spf} (A_{\alpha})
\end{equation}
where $A_{\alpha}$ is a Noetherian adic admissible ring. In particular it is $I=I_{\alpha}$ adic where $I_{\alpha}$ is the maximal ideal of definition of $A_{\alpha}$.
\end{convention}

Another, albeit unrelated, convention that we will adopt is,

\begin{convention}\thlabel{formal:con2}
Following \cite{ega0}[7.6.15], for $S$ a multiplicative subset of an admissible ring,
\begin{equation}
\label{formal5}
A_{\{S\}} = \varinjlim_{f \in S} A_{\{f\}} \, , \quad A \{ S^{-1} \} = \underset{f \in S}{\widehat{\varinjlim}} \, A_{\{f\}}
\end{equation}
wherein $A_{\{f\}}$ is the completion of $A_f$, and $\widehat{\lim}$ is the direct limit in the admissible category, so the latter is the completion of the former in \eqref{formal5}, cf. \cite{formal}[2.5]. Irrespectively if ${\mathfrak p}$, or $x$, or whatever, is an open prime and $S = A \backslash {\mathfrak p}$ then we'll put,
\begin{equation}
\label{formal6}
A_{\{{\mathfrak p}\}} = A_{\{x\}} := A_{\{S\}} \, , \quad A\{{\mathfrak p}\} = A\{x\} = A \{S^{-1}\} \, .
\end{equation}
\end{convention}

The same, or similar, notation is employed for the analogue of polynomials over a ring, i.e.

\begin{revision}\thlabel{formal:rev1}
The ring of $m$-dimensional restricted power series, $A \{T_1 , \cdots , T_m\}$, over an admissible ring is the completion of the polynomial ring $A[T_1 , \cdots , T_m]$ considered as a direct limit of finite free topological $A$-module, \cite{ega0}[7.5.1], while a continuous $A$-algebra, $A \to B$ is of finite type, 
\cite{ega1}[Proposition 10.13.1],
if, for some $m$, there is a topological factorisation,
\begin{equation}
\label{formal7}
\xymatrix{
A \ar[r] \ar[rd]&A \{T_1 , \cdots , T_m\} \ar[d] \\
&B
}
\end{equation}
In particular, therefore, 
the restricted power series are the global functions on $\mathbb{A}^m_{\mathbb{Z}}\hat{\times}\mathrm{Spf}(A)$,
while $B$ has the induced topology afforded by the arrow on the right of \eqref{formal7}.
\end{revision}

\subsection{Formal Henselisation}\label{SS:hensel}

The difference between the localisation $A_f$ and its formal variant $A_{\{f\}}$ notwithstanding, the formal spectra of the latter are a basis of the topology of ${\rm Spf} (A)$. As such any open subset of an affine formal scheme is algebraisable, and, mercifully, this agreeable situation continues in the \'etale topology, to wit:

\begin{fact}\thlabel{hensel:fact1}
Let $\pi : V \to S$ be a quasi-finite map of finite type, \thref{formal:rev1}, between Noetherian formal schemes such that at $x$ over $s$, $\pi$ is flat and un-ramified then for a suitably small Zariski open neighbourhood $V'$ of $x$ there is an open affine neighbourhood $S' = {\rm Spf} (A)$ of $s$, a finite $A$-scheme $X$ and a Zariski open subset $U$ of $X$ \'etale over $S$ such that $V'$ is the completion of $U$.
\end{fact}

\begin{proof}
We don't need adic, just admissible. In any case let $I$ be an ideal of definition, then modulo $I$, $\pi$ is just an \'etale map of schemes and this is described locally by \cite{sga1}[Th\'eor\`eme 7.6]. Specifically for a sufficiently small affinoid neighbourhood $S' = {\rm Spf} (A)$ of $s$, there is a monic polynomial, $f (T)$, such that, around $x$, $V \mod I$ is isomorphic to an open subset of the spectrum of,
\begin{equation}
\label{hensel1}
(A/I) [T] / f(T)
\end{equation}
where of course $f'(T)$ is non-zero $\mod I$ at $x$. Now plainly we can just lift the datum \eqref{hensel1} to get a finite $A$-algebra $X$ and a Zariski open subset $U$ in the same whose reduction $\mod I$ is isomorphic to $V$. This is, however, sufficient, \cite{sga1}[Th\'eor\`eme 8.3], to conclude that the completion of $U$ is isomorphic to a Zariski open neighbourhood of $x$.
\end{proof}

At the risk therefore of an abus de language let us make,

\begin{factdef}\thlabel{hensel:factdef1}
By a formally \'etale map of formal schemes on algebraic spaces is to be understood a quasi-finite, flat and un-ramified map of finite type, \thref{formal:rev1}. In particular if $X = {\rm Spf} (A)$ is an affine formal scheme, and $x : {\rm Spec} \, \overline k \to X$ a map from the separable closure of a point of $X$, the direct limit,
\begin{equation}
\label{hensel2}
A^h_{\{x\}} := \varinjlim_{x \in B} B
\end{equation}
taken over formally \'etale adically complete $B$ through which $x$ factors is a (strictly) Henselian local ring.
\end{factdef}

\begin{proof}
For any $A$-algebra $B$ in the limit \eqref{hensel2}, the point $x$ is the algebraic closure of the residue field of some open prime ideal ${\mathfrak p}$. By definition however $A \to B_{\{{\mathfrak p}\}}$, \thref{formal:con2}, belongs to the limit \eqref{hensel2}, so, without loss of generality, op. cit is a direct limit of local rings whose residue fields have separable closure $\overline k$. As such $A^h_{\{x\}}$ is a local ring which contains the strict Henselisation of $A_{\{x\}}$ by \thref{hensel:fact1}, so it's residue field is $\overline k$. Now let $C$ be a finite $A^h_{\{x\}}$ algebra then for a sufficiently fine $B$ in the limit \eqref{hensel2}, $C$ is actually a finite $B$ algebra and we require to prove,
\begin{equation}
\label{hensel3}
\pi_0 (C \otimes_B \overline k) \longrightarrow \pi_0 (C \otimes_B A^h_{\{x\}})
\end{equation}
is an isomorphism, wherein we identify the corresponding rings with their spectra. Such an isomorphism already holds, however, after tensoring with the strict Henselisation of $B$, which is contained in $A^h_{\{x\}}$.
\end{proof}

The satisfactory nature of \thref{hensel:factdef1} notwithstanding, we have the same ambiguity as \thref{formal:con2}, which in turn we extend to our current situation, i.e.

\begin{convention}\thlabel{hensel:con1}
Let $X = {\rm Spf}(A)$ be an adic Noetherian formal scheme, with $x : {\rm Spec} \, \overline k \to X$ the separable closure of a point then for $B$ as in \eqref{hensel2}, and $\widehat{\lim}$ as immediately post \eqref{formal5},
\begin{equation}
\label{hensel4}
A^h_{\{x\}} := \varinjlim_{x \in B} B \hookrightarrow A^h \{x\} := \underset{x \in B}{\widehat{\varinjlim}} \, B
\end{equation}
i.e. the latter is the $I$-adic completion of the former. As such $A^h \{x\}$ is a local ring, which is Henselian because $A^h_{\{x\}}$ is $\mod I$, and $\pi_0$ of $I$-adically complete rings is determined $\mod I$, whence the obvious variant of \eqref{hensel3}.
\end{convention}

Before concluding we may usefully make,

\begin{rmk}\thlabel{hensel:rmk1}
In the notations of \thref{hensel:con1}, if we confuse $x$ with its image in $X$ then in the notations of \thref{formal:con2}, $A^h_{\{x\}}$ of \eqref{hensel4} will, almost inevitably, be much larger than the Henselisation, $(A_{\{x\}})^h$, of \eqref{formal6}. If, however,  $x$ is a generic point of the trace then they're the same because \'etale neighbourhoods of generic points may be supposed finite of their image when computing generic Henselisation, cf. \eqref{hensel1} et seq.
\end{rmk}

\subsection{Quasi-finite morphisms}\label{SS:quasi}

Formally of finite type, \thref{formal:rev1}, is often just as useful
as finite type, for example:

\begin{fact}\thlabel{quasi:fact1}
Let $A\to B$ be a map of finite type between $I$-adic Noetherian
rings, such that $A/I\to B/I$ is finite, then $A\to B$ is finite.
\end{fact}
\begin{proof}
Let $t_i\in B$ be the images in $B$ of the $T_i$ of \eqref{formal7},
then by hypothesis there is a $d>0$ and polynomials $Q_i$ of degree
at most $d-1$ such that,
\begin{equation}\label{quasi1}
t_i^d \,=\, Q_i(t_i)\quad\mathrm{mod}\, (I)
\end{equation}
and whence the monomials,
\begin{equation}\label{quasi2}
T^D:=T_1^{d_1}\cdots T_n^{d_n},\quad 0\leq d_i< d
\end{equation}
generate $B/I$ over $A/I$. As such, a straightforward induction modulo $I^k$,
shows that the submodule $M$ of $A\{T_1,\hdots , T_m\}$ generated by the $T^D$
of \eqref{quasi2} is onto. Now, \eqref{formal7} is a topological surjection,
so the surjection $M\twoheadrightarrow B$ is also topological, while the induced topology on
$M$ is that of a finite topologically free $A$-module.
\end{proof}

It's simplicity notwithstanding, we should make,

\begin{rmk}\thlabel{rmk:quasi1}
The proof of the related $\mathfrak{m}$-adic \cite{ega2}[Proposition 6.2.5] is
wrong, since the cross referenced \cite{ega0}[7.4.3] comes from op. cit. 7.2.12
which is missing the sufficient finiteness condition of op. cit. corollaire 7.1.14.
On the other hand, these latter assertions are about modules, whereas for adic rings,
finite type, \thref{formal:rev1},  is enough.
\end{rmk}

In turn this affords the best possible local description of,

\begin{fact}\thlabel{fact:quasi2}
Let $f:X\to S$ be a map of finite type, \thref{formal:rev1}, between $I$-adic Noetherian
formal schemes which is quasi-finite, i.e. its trace is quasi finite, then 
for every $s\in S$ there is an \'etale neighbourhood $V\to S$ of $s$ such 
that the fibre $f:X\mid_V\to V$ decomposes as,
\begin{equation}\label{quasi3}
\overline{X}\coprod X'
\end{equation}
where $\bar{f}:=(f\mid_{\overline{X}}):\bar{X}\to V$ is finite, and $(f\mid_{X'})^{-1}(s)$ is empty.
\end{fact}
\begin{proof}
The corresponding proposition is true at the level of the trace, so
denoting the same by the subscript $0$, say $V_0\to S_0$ an \'etale
neighbourhood of $s$ such that \eqref{quasi3} holds with subscript 0.
On the other hand, \eqref{hensel1}, $V_0$ certainly lifts to $V\to S$,
while the connected components of any formal scheme are identically
those of the trace, so we have \eqref{quasi3} with the further proviso
that $\bar{f}_0$ is finite, so $\bar{f}$ is too 
by \thref{quasi:fact1}.
\end{proof}

This allows us to clear up any ambiguity in,
\begin{factdef}\thlabel{fact:quasi3}
Let $f:X\to S$ be a map of finite type, \thref{formal:rev1}, between adic Noetherian
formal schemes then for $X\leftarrow x\mapsto s\rightarrow S$ separably
closed points, we say that $f$ is net at $x$ if any of the following 
equivalent conditions are satisfied:

a) There are \'etale neighbourhoods $U\ni x$, resp. $V\ni s$ such
that $f:U\to V$ is a closed embedding.

b) The map $f^*:\mathscr{O}_{S,s}\to \mathscr{O}_{X,x}$
of strictly Henselian local rings, \eqref{hensel2}, is a topological
surjection.

c) The map $f^*:\mathscr{O}_{S,s}\to \mathscr{O}_{X,x}$
of strictly Henselian local rings, \eqref{hensel2}, is a
surjection.

d )The map $f^*:\mathscr{O}_{S}\{s\}\to \mathscr{O}_{X}\{x\}$
of complete strictly Henselian local rings, \eqref{hensel4}, is a
topological surjection.

e)  The map $f^*:\mathscr{O}_{S}\{s\}\to \mathscr{O}_{X}\{x\}$
of complete strictly Henselian local rings, \eqref{hensel4}, is a
surjection,

f) The trace of $f$ is net at $x$.

g) $x$ is isolated in the fibre $f^{-1}(s)$.

\end{factdef}
\begin{proof}
Plainly (a) implies everything, while all of (a)-(f) imply (g),
so it suffices to do (g) implies (a). However by \thref{fact:quasi2},
or more correctly a minor variant thereof, we may, without loss of
generality, suppose that $X\to S$ is finite, so, inter alia,
algebraisable, with $x,s$ embeddings, and whence, by upper semi-continuity
of the fibres, $X$ is a closed embedding at $x$.
\end{proof}

Needless to say, therefore, it's important to guarantee 
finite type, but this is easy, i.e.

\begin{revision} (cf. \cite{ega1}[Proposition 10.13.1])\thlabel{fact:quasi4}
Let $A$ be complete $I$ adic Noetherian; $B$ an admissible
$A$-algebra with a countable basis of zero such that,
\begin{equation}\label{quasi6}
\begin{split}
A=&\varprojlim_n A_n=A/I^{n+1}\rightarrow \varprojlim_n B_n=B,\,\text{$B_n$ discrete, $n\geq 0$}\\
0\to & B_N I^{n+1} \longrightarrow B_N\longrightarrow B_n\to 0, \, N\geq n,\,\text{is exact,}
\end{split}
\end{equation}
with $A_0\to B_0$ of finite type then
$A\to B$ has finite type in the sense of \thref{formal:rev1}.
\end{revision}
\begin{proof}
For $m\gg 0$, put $C=A[T_1,\hdots T_m]$, and
choose $t_i\in B$, $1\leq i\leq m$, so that we have a
surjection,
\begin{equation}\label{quasi33}
C\otimes_A A_0\twoheadrightarrow B_0: T_i\mapsto t_i.
\end{equation}
Now for every $n\geq 1$ we have a pair of exact sequences,
\begin{equation}\label{quasi4}
\begin{CD}
0@>>> CI^n/CI^{n+1} @>>> C\otimes_A A_n @>>> C\otimes_A A_{n-1}@>>>0 \\
@. @VVV @V{T_i\mapsto t_i}VV @VVV @. \\
0@>>> BI^n @>>> B_n @>>> B_{n-1}@>>>0 
\end{CD}
\end{equation}
wherein the leftmost vertical is surjective by \eqref{quasi4},
which is equally the starting point of proving that the
central vertical is surjective by induction on $n$. As such,
if we
consider the resulting Mittag-Leffler condition,
for the short exact sequences,
\begin{equation}\label{quasi5}
0\to K_n \to C/I^n \xrightarrow{T_i\mapsto t_i}  B/J^n\to 0
\end{equation}
then putting together \eqref{quasi5} with \eqref{quasi4},
affords (ML).
\end{proof}

\subsection{Incoherence of reduction}\label{SS:nil}

Observe that if $X$ is a quasi-compact locally Noetherian formal scheme, adic or otherwise,

\begin{triviality}\thlabel{triv:nil1}
There is a sheaf of ideals,
\begin{equation}
\label{nil1}
U \mapsto {\mathcal N} (U) = \mbox{nil radical of} \ {\mathcal O}_X (U) \, .
\end{equation}
\end{triviality}

\begin{proof}
By definition, if we have a cover $\underset{\alpha}{\coprod} \, U_{\alpha} \to U$ then nilpotents over $U_{\alpha}$ will glue to a function, $g$, on $X$ which is everywhere locally nilpotent, and since $U$ itself is quasi-compact under our hypothesis on $X$, $g$ is globally nilpotent.
\end{proof}

Nevertheless, unlike the case of schemes,

\begin{fact}\thlabel{fact:nil1}
The sheaf ${\mathcal N}$ of \thref{triv:nil1} need not be coherent on adic Noetherian formal schemes.
\end{fact}

We will construct an explicit counterexample following the ideas in \cite{nagata}[A1. Example 3]. As such recall,

\begin{revision}\thlabel{rev:nil1}
(op. cit (E3.1)--(E3.2)) Let $k$ be a field of characteristic $p \ne 0$ such that $(k:k^p)$ is infinite, and $t$ an indeterminate then,
\begin{equation}
\label{nil2}
R := k^p [[t]][k] \hookrightarrow R^* := k[[t]]
\end{equation}
is a  D.V.R. with completion $R^*$. In particular a formal series,
\begin{equation}
\label{nil3}
c(t) = \sum_{i=0}^{\infty} c_i \, t^i \, , \quad c_i \in k
\end{equation}
belongs to $R$ iff $k^p (c_i \mid 0 \leq i \leq \infty)$ is a finite extension of $k^p$. As such if we choose the coefficients so that this extension is infinite, then,
\begin{equation}
\label{nil4}
A := R[c] \xrightarrow{\sim} R[Y]/(Y^p-c^p)
\end{equation}
is a domain whose integral closure $\widetilde A$ in its field of functions $K$ is not a finite $A$-module.
\end{revision}

The proof in op.cit. that $\widetilde A$ over $A$ is infinite is non-constructive. We can, however, usefully observe the matter explicitly. To wit, for every $n \geq 0$ we can write,
\begin{equation}
\label{nil5}
c(t) = g_n (t) + t^n \gamma_n (t)
\end{equation}
where $g_n$ is a polynomial of degree at most $n$, and $\gamma_n (0) = c_n$, so:
\begin{equation}
\label{nil6}
\gamma_n (t) \in K \qquad \mbox{and} \qquad \gamma_n (t)^p = \sum_{j \geq 0} c_{j+n}^p \, t^{pj} \in R \subseteq A \, .
\end{equation}
Better still if $r \in R$ such that $r \gamma_n \in A$ then $t^n \vert \, r$ in $R$. Indeed by \eqref{nil4}, $A$ is a free $R$ module on the basis $c^i$, $0 \leq i \leq p-1$, so an inspection of the coefficient of $c$ in \eqref{nil5} suffices.

\smallskip

Now to get  \thref{fact:nil1} let $d_n$ be a sequence going to infinity with a rate of growth to be determined, and consider the formal series,
\begin{equation}
\label{nil7}
\Gamma (s) := \sum_{n=0}^{\infty} \gamma_{d_n} (t) s^n \in K[[s]] \cap R^* [[s]] 
\end{equation}
in an indeterminate $s$. We assert,

\begin{claim}\thlabel{claim:nil1}
If $L$ is the quotient field of $A[[s]]$,
\begin{equation}
\label{nil8}
X^p - \Gamma (s)^p \in R[[s]][X]
\end{equation}
is an irreducible polynomial in $L[X]$ provided that for every $N > 0$, the sequence
\begin{equation}
\label{nil9}
n^{-N} d_n
\end{equation}
restricted to $n$ a $p^{\rm th}$ power is unbounded.
\end{claim}

\begin{proof}
Following \cite{nish}[5.2.4], suppose otherwise, and write,
\begin{equation}
\label{nil10}
\Gamma(s) = \frac{b (s)}{a(s)}
\end{equation}
where, without loss of generality, all of $\Gamma (0)$, $b(0)$, and $a(0)$ are non-zero. As such from,
\begin{equation}
\label{nil11}
\sum_{i+j=k} a_i \, \gamma_{d_j} = b_k
\end{equation}
for $a_j$, resp. $b_j$, the coefficients of $a(s)$, resp. $b (s)$, so that by induction
\begin{equation}
\label{nil12}
a(0)^n \gamma_{d_n} \in A \, , \quad \forall \, n \geq 0
\end{equation}
and whence, as remarked immediately post \eqref{nil6}, $t^{d_n}$ divides $a(0)^n$ in $R$ if $n$ is a $p^{\rm th}$ power.
\end{proof}

Better still, continuing to follow \cite{nish}, but now at 5.2.2,

\begin{claim}\thlabel{claim:nil2}
The element $X^p - \Gamma (s)^p$ under the hypothesis of \thref{claim:nil1} generates the prime ideal, $P$, defined by,
\begin{equation}
\label{nil13}
0 \to P \to A[[s]][X] \to \frac{L[X]}{X^p - \Gamma^p} \, .
\end{equation}
\end{claim}

\begin{proof}
As a subset of $R^* [[s]]$, $P$ is equally,
\begin{equation}
\label{nil14}
(X-\Gamma (s)) \cap A[[s]][X]
\end{equation}
so $P^p \subseteq (X^p - \Gamma (s)^p)$, and whence $P$ is the only minimal prime of $(X^p - \Gamma (s)^p)$, and is even the only associated prime since $A$ is Cohen Macaulay. The localisation,
\begin{equation}
\label{nil15}
A[[s]][X]_P
\end{equation}
is, however, a localisation of $L[X]$ by \thref{claim:nil1}, so, by op. cit. $X^p - \Gamma (s)^p$ is prime.
\end{proof}

From which it follows,

\begin{cor}\thlabel{cor:nil1}
Under the hypothesis of \thref{claim:nil1} the domain, in the notation of \thref{claim:nil2},
\begin{equation}
\label{nil16}
{\mathcal O} := A[[s]][X]/P
\end{equation}
is a finite $A[[s]]$ module, and whence ${\mathcal O} = \underset{n}{\varprojlim} \, {\mathcal O}/(s^n)$ is an adically complete Noetherian domain.
\end{cor}

At which point we may easily conclude,

\begin{cor}\thlabel{cor:nil2}
On the affine formal scheme ${\rm Spf} {\mathcal O}$ (in the $s$-adic topology) the nil radical, ${\mathcal N}$, of \eqref{nil1} is not coherent.
\end{cor}

\begin{proof}
If ${\mathcal N}$ were coherent then its zero by \thref{cor:nil1}, while by \eqref{nil5},
\begin{equation}
\label{nil17}
\gamma_n (t) = \frac{c - g_n (t)}{t^n} \in A_t \, , \quad n \geq 0
\end{equation}
and no element of $A$ is topologically nilpotent in the $s$-adic topology of ${\mathcal O}$, so we have a formal open affinoid,
\begin{equation}
\label{nil18}
{\rm Spf} {\mathcal O} \{t^{-1}\} \hookrightarrow {\rm Spf} \{{\mathcal O}\} \, .
\end{equation}
The former ring is, however,
\begin{equation}
\label{nil19}
\frac{A_t [[s]][X]}{(X^p - \Gamma (s))^p}
\end{equation}
which is not reduced.
\end{proof}

\subsection{Consistency}\label{SS:con}

Plainly we need to put hypothesis on the formal category in order to avoid un-geometric phenomenon such as \thref{fact:nil1}, and to this end we propose,  

\begin{defn}\thlabel{con:defn1}
Let $X$ be a formal algebraic space or Deligne Mumford champ, and $x$ a separably closed point of $X$ with local ring ${\mathcal O}_{X,x}$ (i.e. $A_{\{x\}}^h$ in the affine case of \thref{hensel:con1}) then we say $X$ is consistent at $x$ if for every (separably closed) generisation $\eta \leadsto x$ which we identify with the prime ideal of ${\mathcal O}_{X,x}$ it defines, every fibre of,
\begin{equation}
\label{con1}
({\mathcal O}_{X,x})_{\eta} \longrightarrow {\mathcal O}_{X,\eta}
\end{equation}
is a geometrically regular ring. Similarly (and it's all we'll actually need) we'll say that $X$ is consistent in a Japanese way at $x$ if every fibre of \eqref{con1}, for every generisation, is geometrically $R_0$ and $S_1$, i.e. reduced. Plainly $X$ is consistent, resp. consistent in a Japanese way if it is so everywhere, 
and we say that $X$ is universally consistent, resp. universally consistent in a Japanese way, 
if every $X'/X$ of finite type, \thref{formal:rev1}, is consistent, resp. consistent
in a Japanese way.
\end{defn}

Before preceding we may usefully give some,

\begin{examples}\thlabel{con:ex1}
If the topology is discrete, i.e. an algebraic space or champ in the usual sense, then the right hand side of \eqref{con1} is the strict Henselisation of the left and $X$ is a lot better than consistent since the geometric fibres are (separable) points.

\smallskip

In the same vein if all the (\'etale) local rings of $X$ are $G$-rings, e.g. excellent, then $X$ is consistent. Indeed by definition, in the diagram
\begin{equation}
\label{con2}
\xymatrix{
({\mathcal O}_{X,x})_{\eta} \ar[rr]^{\rm strict}_{\rm Hensel} &&({\mathcal O}_{X,x})_{\eta}^h \ar[r] \ar[rd] &{\mathcal O}_{X,\eta} \ar[d] \\
&&&\widehat{({\mathcal O}_{X,x})_{\eta}}
}
\end{equation}
where $\wedge$ is completion in $\eta$, the diagonal arrow is regular, so the composition of horizontals is too by \cite{matsumura}[Theorem 32.1].

\smallskip

Similarly if all the (\'etale) local rings are universally Japanese, then $X$ is consistent in a Japanese way -- just change the reference to 23.7 in the proof of op. cit. to 23.9.
\end{examples}

As such consistency covers all practical examples, and better still,

\begin{fact}\thlabel{con:fact1}
Let $X$ be a quasi-compact locally Noetherian formal algebraic space, or Deligne Mumford champ, which is consistent in a Japanese way then the sheaf ${\mathcal N}_X$ of nil radicals, \thref{triv:nil1}, is coherent.
\end{fact}

\begin{proof}
Denoting by the trace by the subscript $0$ we will prove by induction on $d$ that there exists an open $U^{(d)} \subseteq X$ with the following properties,

\begin{enumerate}
\item[(1)] The co-dimension of $Z_0^{(d)} := X_0 \backslash U_0^{(d)}$ is everywhere locally at least $d$.
\item[(2)] There are finitely many, depending on $d$, points $\eta_i \in U_0^{(d)}$ of co-dimension everywhere locally at most $d$ such that, for all $V \to U^{(d)}$ open,
\begin{equation}
\label{con3}
{\mathcal O}_X (V) \hookrightarrow \prod_{\eta_i \in V} {\mathcal O}_{X,\eta_i}
\end{equation}
wherein, and throughout, $\in$ is a shorthand for factors through.
\item[(3)] The sheaf ${\mathcal N}_{U^{(d)}}$ is coherent.
\end{enumerate}
Let $I$ be the ideal of $X_0$ and consider the cone,
\begin{equation}
\label{con4}
\pi : C = {\rm Spec} \coprod_{n \geq 0} I^n / I^{n+1} \to X_0
\end{equation}
then for $\eta_i$ the generic points of $X_0$ we can find an open $U_0 \subseteq X_0$ such that for any open $V$ in $U_0$,
\begin{equation}
\label{con5}
{\mathcal O}_C(V_0) \hookrightarrow \prod_{\eta_i \in V_0} {\mathcal O}_{C,\eta_i}
\end{equation}
which in turn implies that we have an injection,
\begin{equation}
\label{con6}
{\mathcal O}_X (V) \hookrightarrow \prod_{\eta_i \in V_0} {\mathcal O}_X \{ \eta_i \}
\end{equation}
for any open $V$ with trace in $U_0$, so \eqref{con2} holds a fortiori. Plainly, we can, on a sufficiently small open $U \to X$, find a coherent sheaf of ideals ${\mathcal M}$, whose image under \eqref{con3} is the nil radical of the right hand side, and since op. cit. is injective, ${\mathcal M} \hookrightarrow {\mathcal N}_U$. Conversely, ${\mathcal N}$ is sheaf, so ${\mathcal N} (V)$ maps to ${\mathcal M}$ under \eqref{con3}, whence, again by injectivity of op. cit., they're the same, and we've done the case $d=0$.

To go from $d+1$ to $d$ is similar. Specifically let $\eta_i$ be the points of co-dimension at most $d$ which we've already found in item (2), and let $\zeta_j$ be the generic points of $Z_0^{(d)}$. Now for each $j$, reason as above per \eqref{con4}--\eqref{con6} to find open neighbourhoods $V_j$ of each $\zeta_j$ and a coherent sheaf of ideals ${\mathcal M}_j \hookrightarrow {\mathcal N}_{V_j}$ whose image in,
\begin{equation}
\label{con7}
V \longmapsto \prod_{\zeta_j \in V} {\mathcal O}_{X,\zeta_j}
\end{equation}
is the nil radical. In particular, therefore, ${\mathcal M}_j$ generates the nil radical of each,
\begin{equation}
\label{con8}
({\mathcal O}_{X,\zeta_j})_{\eta_i}
\end{equation}
wherein we confuse $\eta_i$ with the prime ideal it defines in ${\mathcal O}_{X,\zeta_j}$. By hypothesis, however, $X$ is consistent in a Japanese way, so ${\mathcal M}_j$ equally generates the nil radical of ${\mathcal O}_{X,\eta_i}$ whenever $\eta_i \leadsto \zeta_j$, so by yet another application of the injectivity of \eqref{con3}, ${\mathcal M}_j$ is equal to ${\mathcal N}$ on $V_j \times_X U^{(d)}$, and whence ${\mathcal M}_j = {\mathcal N}_{V_j}$ is coherent on $V_j$ not just in the Zariski, but also \'etale topology because the consistency condition is \'etale local.
\end{proof}

To which we have a useful,

\begin{cor}\thlabel{con:cor1}
Let $\pi : X \to S$ be a smooth map of $I$-adic affine Noetherian formal schemes (i.e. $\pi$ is of finite type, \thref{formal:rev1}, and modulo $I^{n+1}$ each $\pi_n$ is smooth) then ${\mathcal N}_X = \pi^* {\mathcal N}_S$, so, in particular, if $S$ is consistent in a Japanese way then ${\mathcal N}_X$ is coherent.
\end{cor}

In so much as the topology isn't discrete, we need the obvious,

\begin{lem}\thlabel{con:lem1}
Let $\pi : X \to S$ be as above in \thref{con:cor1}, with $s : {\rm Spec} \, \overline k \to S$ a separably closed point and $x \in X_s (\overline k)$ then there is a formally \'etale, \eqref{hensel:factdef1}, $V \to S$ enjoying a section of $\pi_Y : X_V = X \times_S Y \to V$.
\end{lem}

\begin{proof}
By \cite{sga1}[Th\'eor\`eme 8.5] there is such a neighbourhood $V$ such that the reduction mod $I$, $V_0$, enjoys a section $\sigma_0$ through $x=x_0$. As such proceeding by induction on $n$, the obstruction, cf. \thref{alg:claim0}, to lifting $\sigma_n$ to a section $\sigma_{n+1}$ lies in,
\begin{equation}
\label{con9}
{\rm Ext}_X^1 (\Omega^1_{X/S} , I^{n+1})
\end{equation}
which is zero since $X$ is affine and $\Omega_{X/S}$ is a bundle, so we get a section $\underset{n}{\varprojlim} \, \pi_n$ over our $I$-adically complete $V$.
\end{proof}

We can therefore give,

\begin{proof}[Proof of \thref{con:cor1}] 
Passing to the limit, the section defines a quotient rendering commutative,
\begin{equation}
\label{con10}
\xymatrix{
{\mathcal O}_{S,s} \ar[r] \ar[rd]_{\rm id} &{\mathcal O}_{X,x} \ar[d] \\
&{\mathcal O}_{S,s}
}
\end{equation}
Consequently if the section is cut out by functions $T_1 , \cdots , T_d$ then completing in the same we get a commutative diagram,
\begin{equation}
\label{con11}
\xymatrix{
{\mathcal O}_{X,x} \ar@{^{(}->}[r] &{\mathcal O}_{S,s} [[T_1 , \cdots , T_d]] \\
{\mathcal O}_{S,s} \ar[u] \ar[ru]
}
\end{equation}
in which the nil radical of the biggest ring is generated by the smallest, and the horizontal arrow is flat so ${\mathcal N}_{X,x} = (\pi^* {\mathcal N}_S)_x$, is coherent as soon as $S$ is consistent in a Japanese way by \thref{con:fact1}.
\end{proof}

Unsurprisingly, this is wholly pertinent to the
global behaviour of primary decomposition, to wit:

\begin{fact}\thlabel{flatG:fact66}
Let $\mathscr{M}$ be a coherent sheaf of modules on a formal champ in the smooth
topology which is consistent in a Japanese way, then there is a filtration,
\begin{equation}\label{flatG.21}
\mathscr{M}=\mathscr{M}^0\supset \mathscr{M}^1\supset \cdots \mathscr{M}^n\supset \mathscr{M}^{n+1}=0
\end{equation}
such that everywhere locally: every  associated prime of $\mathscr{M}^i/\mathscr{M}^{i+1}$ is minimal amongst
the same, which in turn define the support of $\mathscr{M}^i$, while their union over $i$ is the 
set of associated primes of $\mathscr{M}$.
\end{fact}
\begin{proof}
Consider first the case of a Noetherian ring $A$ with $\mathscr{M}$ a finite $A$-module $M$ 
enjoying minimal associated primes $\mathfrak{p}_i$, $1\leq i\leq r$, then we have a submodule
defined by the exact sequence,
\begin{equation}\label{ass1}
0\rightarrow M'\rightarrow M\twoheadrightarrow M'' \hookrightarrow \prod_i M_{\mathfrak{p}_i}.
\end{equation}
In particular for any $j$, $M'_{\mathfrak{p}_j}=0$, so no $\mathfrak{p}_j$
is an associated prime of $M'$. Similarly if $\mathfrak{q}\in \mathrm{Ass}(M'')$,
then, from the rightmost inclusion in \eqref{ass1}, there exists $i$ such that
$\mathfrak{q}\subseteq \mathfrak{p}_i$, so, \cite{matsumura}[Theorem 6.3], 
$\mathrm{Ass}(M')$, resp.
$\mathrm{Ass}(M'')$, are the non-minimal, resp. minimal, elements of $\mathrm{Ass}(M)$,
and we've
done the affine scheme case by induction. 
To do the affine adic case, however, we have to ensure that the formal
localisation, $M\to M_{\{S\}}$,  cf. \eqref{formal5}, for a finite set $S$,
preserves the filtration \eqref{flatG.21}, albeit, since completion is
faithfully flat around the trace, it will suffice to do the case $n=0$.
To this end, observe that \cite{matsumura}[Theorem 6.4] admits a small
generalisation, to wit:
\begin{claim}\thlabel{flatG:claim66}
If $A$ is a Noetherian ring, and $M$ a finite module with $n=0$ in
\eqref{flatG.21}, then, non-canonically, $M$ admits a filtration with
successive quotients of the form $A/\mathfrak{p}$, $\mathfrak{p}\in\mathrm{Ass}(M)$.
\end{claim}
\begin{proof}
Start as in op. cit., i.e. choose a sub-module $N\xrightarrow{\sim} A/\mathfrak{p}$,
for some $\mathfrak{p}\in\mathrm{Ass}(M)$, but saturate it, i.e. define $M'$ to be
the fibre,
\begin{equation}\label{ass2}
\begin{CD}
\, @. 0 @. 0 \\
@. @VVV @VVV \\
0@>>> M'@>>> N_{\mathfrak{p}} \\
@. @VVV @VVV \\
0@>>> M @>>> \prod_{\mathfrak{q}} M_{\mathfrak{q}} 
\end{CD}
\end{equation}
so, exactly as above, the quotient, $M''$, of the leftmost row in \eqref{ass2}
still has $n=0$, with $\mathrm{Ass}(M'')\subseteq \mathrm{Ass}(M)$, and, whence,
we conclude by Noetherian induction.
\end{proof}
Now, if $A$ is universally consistent in a Japanese way, then, in the above notation,
$A_{S}/\mathfrak{p}$ is reduced, so, from \thref{flatG:claim66} and \cite{matsumura}[Theorem 6.3],
every associated prime of $M_{\{S\}}$ is a minimal prime over $\mathfrak{p}_{\{S\}}$,
for some $\mathfrak{p}\in\mathrm{Ass}(M)$ as soon as $n=0$ in \eqref{flatG.21}.
After which,  the rest is easy: again it suffices to check that the $n=0$ case is preserved 
by base change, which if
$A\rightarrow B$
were \'etale in the discrete topology, is clear
by \thref{flatG:claim66}, so, the general adic
\'etale case follows from what we've just done
and \thref{hensel:fact1}, while
a general smooth base change reduces to 
a restricted power series ring,
\thref{formal:rev1}, over a sufficiently small 
\'etale neighbourhood.
\end{proof}

Before concluding we may usefully make,

\begin{rmk}\thlabel{rmk:jap1}
When in addition to the hypothesis of \thref{flatG:fact66} we can, functorially,
exclude a subset of minimal primes, we can take the right hand side of \eqref{ass1}
to be all associated primes not containing an excluded one, so that the associated
primes of the resulting $\mathscr{M'}\hookrightarrow \mathscr{M}$ are exactly those
lying over an excluded one, while the resulting quotient admits the filtration of 
\eqref{flatG.21}, but now for the remaining associated primes.
\end{rmk}

\subsection{Remarks on flatness}\label{SS:flatR}

Consistent with \thref{formal:con2} and \thref{hensel:con1} there is an a priori ambiguity in the definition of flatness in the formal category. Mercifully, however, under natural hypothesis appropriate to our current considerations, all ambiguity disappears, i.e.

\begin{fact}\thlabel{fact:flatR1} 
Let $A$ be an $I$-adically complete Noetherian ring, and $M$ an $I$-adically complete $A$-module, i.e. $M = \underset{n}{\varprojlim} \, M / I^n M$,  then $M$ is $A$-flat iff the functor $M \widehat{\otimes}_A$ (topological tensor-product) is exact.
\end{fact}

\begin{proof}
On $A_n = A/I^{n+1}$ modules, $\widehat{\otimes}_A$ and $\otimes_A$ coincide so if $\widehat{\otimes}_A$ is exact, $M_n = M/I^{n+1} M$ is flat over $A_n$ and whence $M$ is $A$-flat by \cite{matsumura}[Theorem 22.3].

Similarly if $M$ is $A$-flat, then $M_n$ is $A_n$ flat. Further, \cite{formal}[Claim 2.6], by definition, short exact sequences of topological $A$-modules always satisfy Artin-Rees, i.e. they're equivalent to a system of short exact sequences of $A_n$ modules,
\begin{equation}
\label{flatR1}
0 \to N'_n \to N_n \to N''_n \to 0
\end{equation}
wherein the transition maps from $n+1$ to $n$ are surjective. In particular therefore,
\begin{equation}
\label{flatR2}
0 \to M_n \otimes_{A_n} N'_n \to M_n \otimes_{A_n} N_n \to M_n \otimes_{A_n} N''_n \to 0
\end{equation}
is an exact inverse system with surjective transition maps, whose limit, $M \widehat{\otimes}_A$, is topologically exact.
\end{proof}

This said the only issue that could have arisen was on infinite $A$-modules since,

\begin{rmk}\thlabel{fact:flatR2}
By Artin-Rees finitely generated $A$-modules are topologically finitely presented, so $M \widehat{\otimes}_A$ and $M \otimes_A$ always coincide on the same even if $M$ weren't adic.
\end{rmk}

Now certainly flat modules are torsion free, but conversely,

\begin{factdef}\thlabel{factdef:flatR1}
Let $M$ be a module over a Noetherian ring $A$ then $M$ is universally torsion free, i.e. $M \otimes_A A''$ is $A''$ torsion free for all quotients $A \twoheadrightarrow A''$, iff $M$ is $A$-flat.
\end{factdef}

\begin{proof}
Flat implies universally torsion free is clear. Conversely, starting with the case that $A$ is Artinian then, cf. post \eqref{flatP.6}, there is a largest quotient $A''$ over, which $M$ is flat. Now suppose $A \ne A''$, then the latter is the quotient by a finitely generated ideal $(f_1 , \cdots , f_t) \ne 0$. In particular replacing $A$ by $A/(f_2 , \cdots , f_t)$ we may, without loss of generality suppose $t=1$. However by \cite{matsumura}[Theorem 22.3 (3)], $M$ is $A$-flat iff,
\begin{equation}
\label{flatR3}
M \otimes_A J \xrightarrow{ \ \sim \ } JM
\end{equation}
which, since $t=1$, is iff $M$ is $f_1$-torsion free.

We can, therefore, conclude by induction on the dimension. Indeed, without loss of generality $A$ is a local ring, and we've just done the case of dimension $0$, while to go from dimension $d$ to $d+1$ it suffices to choose a regular element $x$, so $M$ is $A$-flat iff it is $A/x$ flat and $x$-torsion free.
\end{proof}

\subsection{Purity and flatness}\label{SS:pure}

A scheme or algebraic space $X/S$ which isn't pure may well become formally pure after completion, i.e.

\begin{defn}\thlabel{A:formalpure}
Let ${\mathcal F}$ be a coherent module on a formal algebraic space or Deligne-Mumford champ of formally finite type, \thref{formal:rev1}, $X/S$ with trace $X_0/S_0$ then we say that ${\mathcal F}/S$ is formally pure if ${\mathcal F}_0 := {\mathcal F} \times_X X_0 / S_0$ is pure, \cite{ray}[Definition 3.3.3], and we say ${\mathcal F}/S$ is formally universally pure iff every base change (equivalently every base change of ${\mathcal F}_0/S_0$) is pure.
\end{defn}

While, \thref{flatG:rmk1}, there are plenty of flat modules which aren't pure,

\begin{fact}\thlabel{A:formalpureT}
Let everything be as in \thref{A:formalpure}, then in the sense of \thref{flatG:def2}, the non-flatness of ${\mathcal F}$ is pure.
\end{fact}

\begin{proof}
For $T/S$, let $Z(T) \hookrightarrow X_T = X \times_S T$ be the locus, \thref{flatG:bonus1}, where the base change isn't flat, and observe,
\begin{equation}
\label{pure1}
Z(S) \times_S T \hookleftarrow Z(T) \, .
\end{equation}
Consequently if the non-flatness of ${\mathcal F}$ weren't pure, then without loss of generality in establishing the absurd, there is a point $\sigma \in S$ specialising to $s$ such that for all generic points (of which there is at least one) $\zeta$ of $Z(s)$, $\overline{\zeta} \cap X_s = \varphi$.

This said we follow the logic of the proof of \thref{factdef:flatR1}, and aim to prove the absurd by induction on the dimension $d$ of $S$ at $\sigma$. In particular ${\mathcal O}_{S,\sigma}$ is Artinian if $d=0$, while by \eqref{pure1} we can always replace $S$ by something smaller, so, exactly as in op. cit. the flatifier is given by a principle ideal $J$, and $Z(S)$ is equally the support of the sub-module,
\begin{equation}
\label{pure2}
{\rm ann} (J) \hookrightarrow {\mathcal F} \otimes {\mathcal O}_{S,\sigma}
\end{equation}
as such every generic point $\zeta$ of $Z(S)$ belongs to $Ass ({\mathcal F} \otimes {\mathcal O}_{S,\sigma})$, so no $\overline\zeta \cap X_s$ is empty.

Now suppose the dimension is $d+1 > d \geq 0$, and, as per op. cit, choose a regular element $x \in {\mathcal O}_{S,\sigma}$. We thus have 2 cases,
\begin{enumerate}
\item[(1)] If $S' \hookrightarrow S$ denotes the closed sub-scheme $x=0$, and $Z(S') \ne \varphi$, then we're done by induction and \eqref{pure1}.
\item[(2)] As above, but $Z(S') = \varphi$, so again $Z(S)$ is the support of the annihilator in ${\mathcal F}$ of the principle ideal $(x)$, and we conclude as in the $d=0$ case.
\end{enumerate}
\end{proof}

\subsection{Formal modification}\label{SS:mod}

We follow \cite{artin}[\S1] almost to the letter, but apply the motivation of op. cit. Prop. 1.10 slightly differently, to wit:

\begin{defn}\thlabel{A:mod}
Let $\rho : {\mathcal X}' \to {\mathcal X}$ be a map of formal algebraic spaces of finite type, \thref{formal:rev1}, then we say that $\rho$ is a birational modification if there is a closed subspace ${\mathcal Z} \hookrightarrow {\mathcal X}'$ which contains no irreducible component of the latter such that, cf. \cite{artin}[Definition 1.7],
\begin{enumerate}
\item[(1)] The radicals of the Cramer, $C(\rho)$, and Jacobian ideals, $J(\rho)$, op. cit. 1.3--1.6, contain the ideal, ${\mathcal I}_Z$, defining $Z$.
\item[(2)] The ideal, ${\mathcal I}_{\Delta}$, of the diagonal ${\mathcal X}' \to {\mathcal X}' \times_{\mathcal X} {\mathcal X}'$ is annihilated by $p_1^* {\mathcal I}_Z^N$ for some $N \gg 0$.
\item[(3)] Any map $t : T = {\rm Spf} (R) \to {\mathcal X}$ of finite type at its closed point, wherein $R$ is a complete D.V.R. understood as an ${\mathfrak m}$-adic formal scheme lifts to ${\mathcal X}'$.
\end{enumerate}
\end{defn}

In order to clarify the relation with op. cit. Definition 1.7 let us make,

\begin{rmk}\thlabel{A:modRmk}
By op. cit. Prop 1.10, if $\rho$ were the completion of a map of finite type between algebraic spaces $r : X' \to X$ in an ideal $J$ of $X$, then $r$ is birational in an \'etale neighbourhood of the support of $J$ iff $\rho$ satisfies (1)--(3) of \thref{A:mod}. In the context of \cite{artin}, however, ${\mathcal Z}$ of \thref{A:mod} is always an algebraic space whose reduction is the trace of ${\mathcal X}'$. This latter condition is necessary in proving the main algebraisation theorem of op. cit., but isn't relevant to the definition of birational modification per se.
\end{rmk}

Finally for case of reference, we can usefully repeat the remark of op. cit. prior to 1.8 therein, i.e.

\begin{revision}\thlabel{A:artin}
The meaning of the conditions (1)--(3) of \thref{A:mod} are,

\begin{enumerate}
\item[(1)] $\varphi$ is \'etale off ${\mathcal Z}$.
\item[(2)] $\varphi$ is injective off ${\mathcal Z}$.
\item[(3)] $\rho$ is surjective.
\end{enumerate}
\end{revision}

\subsection{Formal blowing up}

Locally, for $A$, as it will be throughout this section, a complete $I$-adic Noetherian ring, 
this is just the $I$-adic completion of the blow up
of the affine scheme $\mathrm{Spec}(A)$. There are,
however, some pleasant peculiarities that merit further
discussion, beginning with,

\begin{factdef}\thlabel{blow:defn1}
By the formal affine cone over $S=\mathrm{Spf}(A)$, is to be 
understood $\mathbb{A}^n_S$ with the $\mathbb{G}_{m,S}$ action,
\begin{equation}\label{blow1}
\mathbb{A}^n_S\widehat{\times}_S\mathbb{G}_{m,S}:T_i\times\lambda\mapsto \lambda T_i
\end{equation}
or, equivalently, the ring of restricted power series
$A\{T_1,\hdots, T_m\}$ of \thref{formal:rev1} in the above
$A\{\lambda,\lambda^{-1}\}$ action, while by an affine formal
cone $C\to S$ is to be understood an $I$-adic $\mathbb{G}_{m,S}$ equivariant
formal sub-scheme. In particular, it must be defined by
finitely many homogeneous polynomials, so, equivalently
it's the $I$-adic completion of a cone (generated in degree $1$)
over $\mathrm{Spec}(A)$.
\end{factdef}

Now what's important about cones is that, like restricted
power series, they have simple topology, i.e.

\begin{factdef}\thlabel{blow:fact1}
Let $A\to B$ be the algebra of a cone, with $B_n$ its (necessarily
complete finite over $A$) homogeneous piece 
of degree $n$, then $B$ is the subspace of,
\begin{equation}\label{blow2}
\prod_{n\geq 0} B_n
\end{equation}
consisting of sequences $(b_n)$ such that for every $k\geq 0$,
$b_n\in I^kB_n$ for $n\gg 0$, depending on $k$.
\end{factdef}
\begin{proof}
\eqref{blow2} is exactly what one gets on completing $\coprod_n B_n$.
\end{proof}

Applying this to the case of blow ups, we have,

\begin{factdef}\thlabel{blow:fact3}
For $J$ an ideal of $A$, the blow up of $S$ in $J$ is the quotient
by $\mathbb{G}_{m,S}$ of the formal cone with $n$th graded piece $J^n$.
In particular, it is the $I$-adic completion of the blow up of $\mathrm{Spec}(A)$
in $J$, and if $S$ is universally consistent in a Japanese way,
\thref{con:defn1}, then any blow up  is reduced as soon as $A$ is.
\end{factdef} 
\begin{proof}
Rather than the blow up, consider first the
corresponding cone $C\to S$, then, by hypothesis, \thref{con:cor1}, both
of these have a coherent nil-radical, and since they're both affine, this
is nil, in either case as soon as $A$ is reduced by \thref{blow:fact1}. Better,
the affines $T_i\neq 0$, in the notation of \eqref{blow1}, are a
$\mathbb{G}_{m,S}$ equivariant cover of the cone without the zero
section, the global sections of each of which has a nil-radical
because $C$ (qua formal scheme) does. As such, the blow up is covered by reduced affinoids,
and $S$ is universally consistent in a Japanese way, so it's reduced.
\end{proof}

\subsection{Formal rank}\label{SS:rank}

A similar problem to that of \thref{A:mod} presents itself when discussing generic freeness, where we have,

\begin{defrev}\thlabel{rank:def1}
Let $A$ be a Noetherian ring, $M$ a finite $A$-module and $r \in {\mathbb Z}_{\geq 0}$ then the following are equivalent,
\begin{enumerate}
\item[(1)] For every minimal prime $x$ of $A$, $M_x \xrightarrow{ \, \sim \, } A_x^{\oplus r}$.
\item[(2)] There is a non-zero divisor $g \in A$, such that $M_g \xrightarrow{ \, \sim \, } A_g^{\oplus r}$.
\item[(3)] ${\rm Fit}_{r-1} (M) = 0$ and ${\rm Fit}_r (M)$ is not supported on a minimal prime.
\end{enumerate}
\end{defrev}

\begin{proof}
It's clear that (2) $\Rightarrow$ (3) $\Rightarrow$ (1), but to go backwards we need,

\begin{lem} \thlabel{rank:lem1}
Let $X_i$, $1 \leq i \leq n$, be the irreducible components of the spectrum, $X$ of a Noetherian ring $A$, and $Z \hookrightarrow X$ a nowhere dense closed set then the (radical) ideal of $Z$ contains a non-zero divisor.
\end{lem}

\begin{proof}
By induction on $n$, with the case $n=1$ being trivial. As such say $n \geq 2$ and $x_i$ non-zero divisors on 
\begin{equation}\label{rank1}
X^i := \bigcup_{j \ne i} X_j \ \mbox{vanishing on} \ Z^i := \bigcup_{j \ne i} (X_j \cap Z)
\end{equation}
then for $I( \ )$ the ideal in $X$ vanishing on whatever,
\begin{equation}\label{rank2}
I(Z^i) \supseteq (x_i) + I(X^i)
\end{equation}
so that, in particular,
\begin{equation}\label{rank3}
I(Z) \supseteq I(Z^1) \, I(Z^2) \supseteq ((x_1) + I(X^1))((x_2) + I(X^2)) \, .
\end{equation}
Now if everything in the right hand side is a zero divisor then $x_2 \mid_{X_2} = 0$ or $x_1 \mid_{X_1} = 0$. However we can always replace $x_i$ by $x_i + y_i$, where $y_i \mid_{X_i} \ne 0$ but $y_i \mid_{X_j} = 0$ for $j \ne i$, so, this is nonsense.
\end{proof}

Returning, therefore, to the proof of \thref{rank:def1}, (1) affords a closed subset $Z_x \underset{\ne}{\subset} \overline{\{x\}} \hookrightarrow {\rm Spec} (A)$ on whose complement $M$ is trivial while \thref{rank:lem1} affords a non-zero divisor, $g$, such that $D(g) \cap Z_x = \varphi$, $\forall \, x$.
\end{proof}

In consequence we therefore have,

\begin{cor}\thlabel{rank:cor1}
If a finite module satisfies the equivalent conditions of \thref{rank:def1} then so does any flat base change by a Noetherian ring.
\end{cor}

\begin{proof}
Being a non-zero divisor is stable under flat base change.
\end{proof}

However, what's better still is,

\begin{cor}\thlabel{rank:cor2}
If $M$ is a finite module over a Noetherian ring $A$ and $B$ a faithfully flat Noetherian $A$-algebra such that $M \otimes_A B$ satisfies the equivalent conditions of \thref{rank:def1} then so does $M$. In particular if $I$ is an ideal of $A$ and the completion $\widehat M / \widehat A$ satisfies the aforesaid conditions, then so does $M_x / A_x$ for every prime in the support of $I$.
\end{cor}

\begin{proof}
Given minimal primes ${\mathfrak p}_i$ of $A$ we have a finite map,
\begin{equation}
\label{rank4}
A \longrightarrow \prod_i A/{\mathfrak p}_i
\end{equation}
which stays finite after tensoring with $B$, so the good property is (3) of \thref{rank:def1}, i.e. if ${\rm Fit}_{r} (M)$ were supported on a component of $A$, it would be supported on a component of $B$.
\end{proof}

We have, therefore, the following largely satisfactory,

\begin{factdef}\thlabel{rank:fact1}
Let ${\mathscr F}$ be a coherent sheaf on a formal algebraic space ${\mathfrak X}$ then the following are equivalent,
\begin{enumerate}
\item[(A)] For every étale cover $U \to {\mathfrak X}$ by affinoids, $\Gamma (U,{\mathscr F}) / \Gamma (U, {\mathscr O}_{\mathfrak X})$ satisfies the equivalent conditions of \thref{rank:def1}.
\item[(B)] There is some étale cover $U \to {\mathfrak X}$ by affinoids such that $\Gamma (U,{\mathscr F})/\Gamma(U,{\mathscr O}_{\mathfrak X})$ satisfies the equivalent conditions of \thref{rank:def1}.
\item[(C)] For every separably closed point $x$ of the trace, the ${\mathfrak m} (x)$-adic completion of ${\mathscr F}_x / {\mathscr O}_{{\mathfrak X},x}$ satisfies the equivalent conditions of \thref{rank:def1}.
\end{enumerate}
and should any of these equivalent conditions occur, we say that ${\mathscr F}$ is generically free of rank $r$.
\end{factdef}

\begin{proof}
(A) $\Rightarrow$ (B) is clear, (B) $\Rightarrow$ (C) by \thref{rank:cor1}, while for any $U \to {\mathfrak X}$ (C) implies, by \thref{rank:cor1} \& \thref{rank:cor2}, that ${\rm Fit}_{r-1} \Gamma (U , {\mathscr F})$, resp. ${\rm Fit}_r (U , {\mathscr F})$ is zero, resp. without support in a component, after localising at every open prime, and since $\Gamma (U,{\mathscr O}_{\mathfrak X})$ is $I$-adically complete these respective conditions hold everywhere.
\end{proof}

Before progressing let us observe,

\begin{bonus}\thlabel{rank:bonus1}
If ${\mathfrak X}$ is reduced and globally irreducible then the equivalent conditions of \thref{rank:fact1} hold iff for some generic point, $\xi$, of the trace they hold for the $I$-adic (equivalently ${\mathfrak m}(\xi)$) completion of ${\mathscr F}_{\xi} / {\mathscr O}_{{\mathfrak X} , \xi}$, where, indeed they do hold for some $r$.
\end{bonus}

To address this we'll need another,

\begin{lem}\thlabel{rank:lem2}
If ${\mathfrak X}$ is a globally irreducible formal algebraic space and ${\mathfrak Y} \hookrightarrow {\mathfrak X}$ a closed subspace such that at some point $x$, ${\mathfrak Y}$ contains an irreducible of ${\mathfrak X}$ then ${\mathfrak X} = {\mathfrak Y}$.
\end{lem}

\begin{proof}
Exactly as in the proof of \eqref{rank4} we identify reducibility with a finite modification, \thref{A:mod},
\begin{equation}
\label{rank5}
{\mathfrak X}' \coprod {\mathfrak X}'' \longrightarrow {\mathfrak X}
\end{equation}
for closed proper subspaces ${\mathfrak X}' , {\mathfrak X}'' \hookrightarrow {\mathfrak X}$. Quite generally, however, for $J$ an ideal in a reduced ring, $A$, we have an injection,
\begin{equation}
\label{rank55}
A \longhookrightarrow A/J \prod A/{\rm ann} (J)
\end{equation}
while for ${\mathscr I}$ the ideal sheaf of ${\mathfrak Y}$, the conditions of the theorem amounts to $({\mathscr I}_x)_{\mathfrak p} = 0$ for ${\mathfrak p}$ a minimal prime of ${\mathscr O}_{{\mathfrak X},x}$, so ${\rm ann} ({\mathscr I})_x \ne 0$, and whence defines a (possibly empty) closed subspace. However by \eqref{rank5} \& \eqref{rank55}, ${\mathscr I} = 0$ or ${\rm ann} ({\mathscr I}) =0$ everywhere.
\end{proof}

We are now in a position to prove,

\begin{proof}[Proof of \thref{rank:bonus1}]
By hypothesis ${\rm Fit}_{r-1} ({\mathscr F})$ is $0$ at $\xi$ so it's zero everywhere by \thref{rank:lem2}. Similarly ${\rm Fit}_r({\mathscr F})$ is non-zero at $\xi$, so it can't be identically zero on any component of any local ring. Better still we only need these hypothesis on one component through $\xi$ to arrive to the same conclusion.
\end{proof}

All of which, more or less, admits an infinitesimal interpretation, to wit:

\begin{defrev}\thlabel{rank:def2}
Let $M$ be a finite module over a (${\mathfrak m}$-adically) complete local ring with maximal ideal $A$, and $F^0$ a good, \cite{ega3}[0.13.7.7], fibration, i.e.
\begin{equation}
\label{alg9}
F^0 M = M \, , \ {\mathfrak m} \, F^p \subseteq F^{p+1} \, , \ \forall \, p \geq 0 \, , \ \mbox{and} \ \exists \, q \geq 0 \ni {\mathfrak m} \, F^p = F^{p+1} \, , \ p \geq q
\end{equation}
then, \cite{matsumura}[Theorem 13.4], its length as an $A_m = A / {\mathfrak m}^{m+1}$ module enjoys asymptotics,
\begin{equation}
\label{rank6}
\ell_{A_m} (M/F^m M) = a m^{d(M)} \left( 1+0 \left( \frac1m \right)\right)
\end{equation}
where $d(M)$ is the dimension of $M$, $a \ne 0$. As such the limit,
\begin{equation}
\label{rank7}
\widehat r_A (M) := \varinjlim_{m \to \infty} \, \frac{\ell_{A_m} (M/F^m M)}{\ell_{A_m} (A_m)}
\end{equation}
exists, and should $A$ be reduced and irreducible, or more generally has constant rank on each component, is exactly the rank at the generic point(s).
\end{defrev}

\subsection{Specialisation and cohomology}\label{SS:count}

These considerations can be applied to prove analogues of the upper semi-continuity and invariance of the Euler characteristic in situations where they don't a priori have sense. Specifically, consider,

\begin{setup}\thlabel{setup:count1}
Let $\pi : {\mathfrak X} \to S$ be a proper map of formal algebraic spaces, ${\mathscr F}$ a coherent sheaf on ${\mathfrak X}$, $s$ a point of $S$, and $\widehat\pi : \widehat{\mathfrak X} \to \widehat S$, $\widehat{\mathscr F}$ their completion in $s$, then, for $q \geq 0$ we define,
\begin{equation}
\label{count1}
\widehat h^q ({\mathfrak X},{\mathscr F}) := \widehat r_{{\mathscr O}_{\widehat S,s}} (R^q \, \widehat \pi_* \, {\mathscr F}) \, , \ \widehat X ({\mathfrak X},{\mathscr F}) = \sum_{q=0}^{\infty} (-1)^q \, \widehat h^q ({\mathfrak X},{\mathscr F}) \, .
\end{equation}
\end{setup}

Plainly we can immediately conclude from \thref{rank:def2},

\begin{fact}\thlabel{count:fact1}
If all $R^q \, \pi_* \, {\mathscr F}$ have generically constant rank $r^q$, \thref{rank:fact1}, e.g. $S$ reduced and globally irreducible, \thref{rank:bonus1}, then, $\widehat h^q ({\mathfrak X},{\mathscr F}) = r^q$.
\end{fact}

\begin{proof}
By the theorem on formal functions $R^q \, \widehat\pi_* \, {\mathscr F}$ is $(R^q \, \pi_* \, {\mathscr F}) \otimes_{{\mathscr O}_S} \widehat{\mathscr O}_{S,s}$, so item (C) of \thref{rank:fact1} applies.
\end{proof}

To fully profit from this observe,

\begin{lem}\thlabel{alg:lem4}
Let everything be as in \thref{setup:count1} with ${\mathscr F}$ a coherent sheaf on ${\mathfrak X}$, then, for all $i \geq 0$,
\begin{equation}
\label{alg31}
\widehat h^q ({\mathfrak X} , {\mathscr F}) = \lim_{n \to \infty} \ \frac{\ell_{A_m} ({\rm H}^i (X_m , {\mathscr F} \otimes_A A_m))}{\ell (A_m)} 
\end{equation}
for $A = \widehat{\mathscr O}_{S,s}$, and $X_m = \widehat{\mathfrak X} \times_{\widehat S} {\rm Spec} \, A/{\mathfrak m}^m =: A_m$.
\end{lem}

\begin{proof}
By \cite{ega3}[corollaire 4.7.1] the filtration defined by,
\begin{equation}
\label{alg31.new.new}
0 \longrightarrow F^{m+1} {\rm H}^i (X,{\mathscr F}) \longrightarrow {\rm H}^i (X,{\mathscr F}) \longrightarrow {\rm H}^i (X_m , {\mathscr F} \otimes_A A_m) \, , \quad i \geq 0
\end{equation}
is ${\mathfrak m}$-good, \thref{rank:def2}. As such if following \cite{ega3}[4.1.7.5] we introduce the quotients,
\begin{equation}
\label{alg32.new}
{\rm H}^i (X,{\mathscr F}) \longrightarrow {\rm H}^i (X_m , {\mathscr F} \otimes_A A_m) \longrightarrow Q_m \longrightarrow 0
\end{equation}
then our goal, \eqref{alg31}, is equivalent to,
\begin{equation}
\label{alg33.new}
\lim_{m \to \infty} \, \frac{\ell (Q_m)}{\ell (A_m)} = 0 \, .
\end{equation}
However by \cite{ega3}[4.1.7.15] the module,
\begin{equation}
\label{alg34.new}
\coprod_{m \geq 0} \, Q_m
\end{equation}
is finite over some $A_r$, $r \gg 0$, so, in fact, the $Q_m$ are eventually zero.
\end{proof}

The utility of this is that it makes $\widehat h^q$ easy to estimate, e.g.

\begin{fact}\thlabel{count:fact2}
Suppose ${\mathscr F}/S$ is flat then,
\begin{enumerate}
\item[(1)] $\widehat{\chi} ({\mathfrak X},{\mathscr F}) = \chi (X_0 , {\mathscr F}_0)$
\item[(2)] $\widehat h^q ({\mathfrak X},{\mathscr F}) \leq h^q (X_0 , {\mathscr F}_0)$.
\end{enumerate}
\end{fact}

\begin{proof}
For every $m \geq 0$, flatness gives an exact sequence,
$$
0 \longrightarrow {\mathscr F}_0 \otimes \frac{{\mathfrak m}^m}{{\mathfrak m}^{m+1}} \longrightarrow {\mathscr F} \otimes_A A_{m+1} \longrightarrow {\mathscr F} \otimes_A A_m \longrightarrow 0
$$
so (1) is immediate, while the long exact sequence in cohomology gives,
\begin{eqnarray}
\ell_A (H^q (X_{m+1} , {\mathscr F} \otimes_A A_{m+1})) &\leq &\ell_A (H^q (X_m , {\mathscr F} \otimes_A A_m)) \nonumber \\
&+ &\left( \dim_{k(s)} \, \frac{{\mathfrak m}^m}{{\mathfrak m}^{m+1}} \right) \cdot h^q (X_0 , {\mathscr F}_0)\nonumber
\end{eqnarray}
so summing over $m$ gives (2).
\end{proof}
